\documentclass[12pt]{amsart}
\usepackage{amsmath}
\usepackage{amsfonts}
\usepackage{amssymb}
\usepackage{epsfig}
\usepackage{epstopdf}
\usepackage{comment}
\usepackage{color, url}
\usepackage[footnote]{fixme}

\definecolor{blue}{RGB}{0,0,255}
\definecolor{green}{RGB}{50,150,50}
\definecolor{red}{RGB}{255,0,0}

\newtheorem{theorem}{Theorem}[section]
  \newtheorem{lemma}[theorem]{Lemma}

   \newtheorem{corollary}[theorem]{Corollary}
  \newtheorem{proposition}[theorem]{Proposition}

  \newtheorem{definition}[theorem]{Definition}
  \theoremstyle{definition}
  \newtheorem{example}[theorem]{Example}
   \newtheorem{remark}[theorem]{Remark}

\newcommand{\CC}{\mathbb C}
\newcommand{\RR}{\mathbb R}
\newcommand{\QQ}{\mathbb Q}
\newcommand{\ZZ}{\mathbb Z}
\newcommand{\NN}{\mathbb N}
\newcommand{\KK}{\mathbb K}

\newcommand{\cT}{\mathcal T}

\newcommand{\kk}{\mathbf k}

\newcommand{\stint}{\cdot}
\newcommand{\minksum}{+}
\newcommand{\cyclesum}{\oplus}
\newcommand{\refl}[1]{-#1} 

\DeclareMathOperator{\supp}{supp}

\DeclareMathOperator{\vol}{vol}

\DeclareMathOperator{\init}{in}

\DeclareMathOperator{\codim}{codim}
\DeclareMathOperator{\link}{link}

\DeclareMathOperator{\Ass}{Ass}
\DeclareMathOperator{\Hom}{Hom}
\DeclareMathOperator{\MV}{MV}
\DeclareMathOperator{\Span}{span}
\DeclareMathOperator{\Gfan}{Gfan}

\newcommand{\mult}[2]{\operatorname{mult}_{#2}({#1})}

\date{\today}
 
 \title[]{Stable Intersections of Tropical Varieties}
 \author{Anders Jensen}
 \author{Josephine Yu}
 \address{Institut for Matematik, Aarhus Universitet, Aarhus, Denmark}
\email{jensen@math.au.dk}
\address{ School of Mathematics, Georgia Institute of Technology,
        Atlanta GA, USA}
\email {jyu@math.gatech.edu}

\date{\today}

\begin{document}

 \begin{abstract}
We give several characterizations of stable intersections of tropical cycles and establish their fundamental properties.  We prove that the stable intersection of two tropical varieties is the tropicalization of the intersection of the classical varieties after a generic rescaling. 
A proof of Bernstein's theorem follows from this.  We prove that the
tropical intersection ring of tropical cycle fans is isomorphic to
McMullen's polytope algebra.  It follows that every tropical cycle fan
is a linear combination of pure powers of tropical hypersurfaces,
which are always realizable.  We prove that every stable intersection of constant coefficient tropical varieties defined by prime ideals is connected through codimension one.  We also give an example of a realizable tropical variety that is connected through codimension one but whose stable intersection with a hyperplane is not.
 \end{abstract}

 \maketitle

\section{Introduction}

We consider stable intersections of tropical cycles in $\RR^n$, which are pure weighted balanced rational polyhedral complexes.  Stable intersection is an important and useful concept that has been studied by many authors.  It first appeared under a different guise in \cite{FultonSturmfels} where Fulton and Sturmfels performed multiplication of Chow cohomology classes on a complete toric variety using a procedure called the {\em fan displacement rule}.  The same procedure was used in \cite{RGST, Mik} to define stable intersection of tropical varieties and cycles.  Allermann and Rau gave a different definition of the intersection product of tropical cycles using tropical Cartier divisors \cite{AllermannRau}.  Katz and Rau then showed independently that these two notions of intersection products coincide \cite{KatzIntersection, Rau}.  Kazarnovskii independently developed a theory of tropical varieties and their stable intersections using a different language~\cite{Kazarnovskii}.

Our interest in stable intersection originated from our previous work on computation of tropical resultants~\cite{tropRes}.  For resultants of polynomial systems, the most ubiquitous cases, including the cases for implicitization, are where some of the coefficients are {\em specialized} to generic constants while we study the choices of remaining coefficients that make the system solvable. When the resultant is defined by a single polynomial, specialization corresponds to projection of the Newton polytope.  In general, for tropical varieties, which can be considered as generalizations of Newton polytopes, specialization corresponds to stable intersection with a coordinate subspace. We were dissatisfied with the lack of efficient algorithms for computing stable intersections as well as the lack of elementary proofs for basic results about stable intersections, especially those relating to perturbation of ideals, that would be accessible to a wider computational geometry community.
\smallskip

 The contributions of the present paper include:
\begin{itemize}
\item A new characterization of stable intersection that enables us to do computations efficiently.  This has been implemented in Gfan~\cite{gfan} and also as a Polymake extension by Hampe~\cite{Hampe}.
\item Self-contained combinatorial proofs of fundamental properties of stable intersection, including well-definedness, dimension formula, balancing condition, and associativity.
\item An elementary proof that the stable intersection of tropical varieties is the tropicalization of the intersection of the varieties after a generic rescaling.
\item A proof that the tropical intersection ring is isomorphic to McMullen's polytope algebra.
\item A proof that the stable intersection of irreducible tropical varieties is connected through codimension one, answering an open question of Cartwright and Payne from~\cite{CartwrightPayne}.
\item An example showing that stable intersection does not preserve connectivity through codimension one, even for realizable tropical varieties.
\end{itemize}
 
We give a new definition of stable intersection (Definition~\ref{def:stableIntersection}) and show that it is equivalent to both the fan displacement rule (Proposition~\ref{prop:limit}) and the Allermann--Rau intersection (Proposition~\ref{prop:AR}).  Our definition is preferable for computations since it does not use limits as in the fan displacement rule and does not require performing intersections iteratively on the Cartesian product as the Allermann--Rau definition does.  In Section~\ref{sec:basics} and in the Appendix we give detailed and careful proofs of fundamental results about stable intersections without relying on algebraic geometry results.

In Section~\ref{sec:ideals} we show that the stable intersection of tropical varieties is the tropicalization of the intersection of varieties after a generic scaling of variables.  This was proven by Osserman and Payne \cite[Proposition 2.7.8]{OssermanPayne}; however our methods are different and  elementary.  From this we derive Bernstein's Theorem in Section~\ref{sec:Bernstein}.  

We show in Section~\ref{sec:PolytopeAlgebra} that the ring of tropical
cycles with the stable intersection product is isomorphic to
McMullen's polytope algebra, based on a result by Fulton and
Sturmfels.  From this and McMullen's work on polytope algebra, it
follows that every tropical cycle is a linear combination of pure
powers of tropical hypersurfaces.  In particular, it follows that
every tropical cycle is a linear combination of realizable tropical varieties.

Tropical varieties of prime ideals are known to be connected through codimension one. We prove that stable intersections of such tropical varieties are also connected through codimension one (Theorem~\ref{thm:connected}).  For arbitrary (non-irreducible) tropical cycles, even for realizable ones, the stable intersection does not preserve connectivity through codimension one (Example~\ref{ex:disconnected}).

\medskip

\noindent
{\bf Notations and conventions.} We use the {\em max} convention in tropical geometry.

Our tropical cycles are polyhedral complexes in Section~\ref{sec:basics}, and we assume that they are fans in Sections~\ref{sec:ideals}, \ref{sec:Bernstein}, \ref{sec:PolytopeAlgebra}, and~\ref{sec:connectivity}.

In many situations throughout the paper we use the notion of a \emph{generic} point $p$ in a set $S$. By this we mean that $p\in S$ is chosen in an implicitly given relatively open dense subset of $S$. A typical situation is when $S$ is the support of a polyhedral complex and generic points are those in the relative interior of facets.

\medskip
\noindent
{\bf Acknowledgments.}  We thank Diane Maclagan for reading and providing feedback on an earlier draft and the referees for many helpful suggestions that greatly improved the exposition. The first author was supported by the Danish Council for Independent Research, Natural Sciences (FNU), and the second author was supported by the NSF grant DMS \#1101289.

 \section{Definitions and Basic Properties}
\label{sec:basics}

Let $N$ be a lattice, $N_\QQ = N \otimes_\ZZ \QQ$ and $N_\RR = N \otimes_\ZZ \RR$.  We may sometimes refer to $N_\RR$ as $\RR^n$ where $n$ is the dimension.

A tropical $k$-cycle in $N_\RR$ is a pure $k$-dimensional weighted balanced rational polyhedral complex as we now explain.  A polyhedral complex $X$ is called {\em weighted} if every facet $\sigma$ is assigned a number $\mult{\sigma}{X}$ which we call its multiplicity or weight. The multiplicity is usually an integer, but we also allow rational multiplicities in Section~\ref{sec:Bernstein} and~\ref{sec:PolytopeAlgebra}.  If a point $x$ is in the relative interior of a facet $\sigma$, then let $\mult{x}{X} := \mult{\sigma}{X}$.  For a face $\sigma \in X$, let $N_{\sigma}$ denote the maximal sublattice of $N$ parallel to the affine span of $\sigma$, and for any relative interior point $\omega \in \sigma$, let $N_\omega := N_\sigma$. 

The {\em support} $\supp(X)$ of a tropical cycle $X$ is the union of its closed facets with non-zero multiplicities. When it leads to no confusion we will sometimes repress the ``supp'' notation. If $\supp(X)=N_\RR$ then $X$ is called \emph{complete}. The {\em link} of a polyhedron $\sigma \subseteq \RR^m$ at a point $v \in \sigma$ is the polyhedral cone
$$
\link_v(\sigma) = \{u \in \RR^m \,|\, \exists \delta > 0 :  \forall \varepsilon \text{ between } 0 \text{ and } \delta : v + \varepsilon u \in \sigma \}.
$$
The {\em link} of a tropical cycle $X$ at a point $v\in\supp(X)$ is the tropical cycle
$$
\link_v(X) = \{ \link_v(\sigma) \,|\, v \in \sigma \in X \}
$$
with inherited multiplicities. This is always a fan.
We also use the notation $\link_\tau(\sigma)$ and $\link_\tau(X)$ to denote links with respect to relative interior points of $\tau$.  The {\em lineality space} of a polyhedron $P$ is the largest affine subspace of $P$, translated to the origin. All cones in a given fan have the same lineality space, which we call the lineality space of the fan.  The lineality space of $\link_\tau(X)$ contains the linear span of $\tau$.

A weighted rational polyhedral complex $X$ is called {\em balanced} if for any ridge (codim-$1$ face) $\tau$ of $X$,
$$
\sum_{\sigma \supset \tau} \mult{\sigma}{X}\cdot v_{\sigma/\tau}  \in \Span_\QQ(N_\tau)
$$
where $\sigma$ runs over facets of $X$ containing $\tau$ and $v_{\sigma/\tau} \in N\cap \link_\tau(\sigma)$ is a lattice element generating $N_\sigma$ together with $N_\tau$. This completes the definition of tropical cycles.

The polyhedral complex structure is disregarded as follows. For a tropical cycle $X$ and a complete complex $Y$, the \emph{common refinement} $X\wedge Y:=\{\sigma\cap\tau:\sigma\in X ,\tau\in Y \}$ inherits the multiplicities of $X$. Two cycles $X$ and $Y$ are identified if there exists a complete complex $Z$ such that $X \wedge Z=Y\wedge Z$ with multiplicity.  Moreover we ignore facets with multiplicity $0$.

In following we will consider images of tropical cycles under linear maps. For simplicity we assume that all polyhedra in a given cycle have the same lineality space. This will not be a restriction later as we are mainly interested in fans. Let $X$ be a tropical cycle in $N_\RR$ and $A : N \rightarrow N'$ be a linear map, inducing a map $A : N_\RR \rightarrow N'_\RR$.  Suppose that for a dense open subset of the image $A(X)$, the preimage of each point consists only of points in the relative interiors of facets in $X$.  In other words, modulo a subspace of the lineality space of $X$, the map $A$ is generically finite-to-one on $X$.  For the rest of the paper, all linear maps between tropical cycles that we consider will satisfy this property.

We can then define multiplicities on the image $A(X)$ as follows.  
First endow $A(X)$ with a polyhedral structure such that the image of each face of $X$ is a  union of faces of $A(X)$.  For any point $\omega \in A(X)$ lying in the relative interior of a facet, let
\begin{equation}
\label{eqn:ST}
\mult{\omega}{A(X)} = \sum_{v} \mult{v}{X} \cdot [N'_\omega : A N_v ],
\end{equation}
where the sum runs over one $v$ for each facet of $X$ meeting the preimage of $\omega$.

If $X$ is the tropical variety of an ideal $I$ in the sense of Section~\ref{sec:ideals} and $A$ is the tropicalization of a map $\alpha$ of tori, we have the relation $A(X)=\delta\cT(\alpha(V(I)))$, where $\delta$ is the degree of $\alpha$ on $V(I)$. This can be seen from the Sturmfels--Tevelev projection formula for the generically finite-to-one case~\cite{SturmfelsTevelev} and its generalization by Cueto--Tobis--Yu~\cite[Theorem~3.4]{CTY}. When $A(X)$ is the entire ambient space, $\cT(\alpha(V(I)))$ has multiplicity one everywhere and $\delta$ is the multiplicity of $A(X)$. However, in general we cannot recover $\delta$ tropically from $A$ and $\cT(I)$ as the following example shows.

\begin{example}
\label{ex:ST}
Let $I=\langle x_1+x_2+x_3+1,(x_3-2)(x_3-1)\rangle$ and $J=\langle x_1+x_2+1,(x_3-2)(x_3-1)\rangle$ be ideals in $\CC[x_1,x_2,x_3]$ and $A:\ZZ^3\rightarrow\ZZ^2$ be the projection to the first two coordinates. Then $\cT(I)=\cT(J)$ consists of the three rays $-e_1,-e_2$ and $e_1+e_2$, each with multiplicity $2$. The ideal $J\cap \CC[x_1,x_2]$ is a linear ideal. This makes the projection of $V(J)$ to the first two coordinates a degree-2 map. The ideal $I\cap \CC[x_1,x_2]$, however, is generated in degree two, making the projection of $V(I)$ a degree-1 map.

In other words, the tropical projection formula computes the push-forward of cycles, rather than the image variety.
\qed
\end{example}


\begin{lemma}
\label{lem:balanced}
 Let $\tau$ be a ridge in $X$ such that $A(\tau)$ also has codimension $1$ in $A (\link_X(\tau))$.  Then $A (\link_X(\tau))$ is balanced with multiplicity defined in~(\ref{eqn:ST}).
\end{lemma}
The proof of this lemma is a computation with lattice indices and is given in the Appendix.

\begin{corollary}
\label{cor:genericwelldefined}
If the support of $A(X)$ contains a full-dimensional polyhedron, then it is all of $N'_\RR$, and the multiplicity given by formula~(\ref{eqn:ST}) is constant on a dense open subset of $N'_\RR$.
\end{corollary}

\begin{proof}
Suppose the support of $A(X)$ is a union of finitely many polyhedra, but it is not all of $N'_\RR$. Then there is a face on the boundary of $\supp(A(X))$ with dimension $\dim(N'_\RR)-1$.  However, the balancing condition cannot hold at this face, contradicting Lemma~\ref{lem:balanced}.  The balancing condition also ensures that any two facets that meet along a codimension one face have the same multiplicity.  Since any two facets are connected by a path that avoids codimension two faces, all facets must have the same multiplicity.
\end{proof}

For subsets $S, T \subset N_\RR$, we will use $\minksum$ for Minkowski sum $S \minksum T = \{x+y : x\in S, y \in T\}$.  Also, let $\refl{S} = \{-x : x \in S\}$  and $S - T = S + (-T)$.  For a tropical cycle $X$, let $\refl{X}$ denote a tropical cycle with $\supp(\refl{X}) = - \supp(X)$ and $\mult{\omega}{\refl{X}}= \mult{-\omega}{X}$. 

Suppose $X$ and $Y$ are tropical cycles and $L$ is a subspace of the linelity spaces of $X$ and $Y$ such that $\dim(X/L) + \dim(Y/L) = \dim ((X+Y)/L)$.  We can define multiplicities on $X \minksum Y$ by applying the formula (\ref{eqn:ST}) to the projection of the Cartesian product $X \times Y$ onto $X \minksum Y$ via $(x,y) \mapsto x+y$. It can be verified straightforwardly that the product $X \times Y$ is a tropical cycle with multiplicity given by $\mult{(u,v)}{X \times Y} := \mult{u}{X} \mult{v}{Y}$. 
More concretely, a generic point $v \in X \minksum Y$ has multiplicity:
\begin{equation}
\label{eqn:MinksumMult}\mult{v}{X \minksum Y} = 
\sum_{\substack{\sigma_1, \sigma_2}}\mult{\sigma_1}{X} \mult{\sigma_2}{Y}[N_v:N_{\sigma_1}+N_{\sigma_2}]\end{equation}
where the sum is over all pairs of facets $\sigma_1 \in X$ and $\sigma_2\in Y$ such that $v \in \sigma_1 \minksum \sigma_2$.
Here the formula~(\ref{eqn:MinksumMult}) works for all $v\in X \minksum Y$ for which the fiber over $v$ in $X \times Y$ contains finitely many points, each of which lies in the relative interior of a facet in $X \times Y$.  The set of all such $v$'s form a dense relatively open subset of $X \minksum Y$.  A similar situation arises in the following definition.

\begin{definition}
\label{def:stableIntersection}
Let $X$ and $Y$ be tropical cycles in $N_\RR$. The {\em stable intersection} $X \stint Y$ is a weighted polyhedral complex defined by:
$$\supp(X \stint Y):=\{\omega\in N_\RR: \supp(\link_\omega X - {(\link_\omega Y)})=N_\RR\}$$ which has a polyhedral complex structure as a subcomplex of $X \cap Y$ with a natural polyhedral complex structure as the common refinement $\{\sigma_1 \cap \sigma_2 : \sigma_1 \in X, \sigma_2 \in Y\}$.  
For a face $\gamma \in X \stint Y$ with $\codim(\gamma) = \codim(X) + \codim(Y)$, the multiplicity $\mult{\gamma}{X \stint Y}$ is defined to be the multiplicity of $\link_\gamma X - {(\link_\gamma Y)}$ given by the formula (\ref{eqn:MinksumMult}) above. 
\end{definition}
We will prove in Theorem~\ref{thm:balanced} that $X\stint Y$ has codimension $\codim(X)+\codim(Y)$, justifying defining a multiplicity of cones $\gamma$ of this codimension.

Since $\link_\gamma X$ and $-\link_\gamma Y$ are tropical cycles, if the support of their Min\-kow\-ski sum $\link_\gamma X - (\link_\gamma Y)$ is all of $N_\RR$, then  by Corollary~\ref{cor:genericwelldefined} the multiplicity function is constant on a dense open subset.  This constant is the multiplicity of $\gamma$ in $X \stint Y$.  
The multipicity of $\link_\gamma X - (\link_\gamma Y)$ is well-defined because
\begin{align*}\dim(\link_\gamma (X)/L)+\dim(\link_\gamma (Y)/L)&=\dim(\link_\gamma X)+\dim(\link_\gamma Y)-2\dim(L)\\&=\dim(X)+\dim(Y)-2\dim(L)\\&=2n-\codim(X)-\codim(Y)-2\dim(L)\\&=2n-\codim(L)-2\dim(L)\\&=n-\dim(L)\\&=\dim(\link_\gamma X-\link_\gamma Y)-\dim(L)\end{align*}
where $L$ is the centered affine span of $\gamma$.

For a generic element $v \in N_\RR$,
$$
\mult{\gamma}{X \stint Y} = \sum_{\sigma,\tau}  \mult{\sigma}{X} \mult{\tau}{Y}[N : N_\sigma + N_\tau]
$$
where the sum is over all pairs of facets $\sigma \in \link_\gamma(X)$ and $\tau \in \link_\gamma(Y)$ such that $v \in \sigma - \tau$ (or equivalently, $\sigma$ intersects $\tau + v$). The formula works for all $v$'s for which the fiber over $v$ under the map $(x,y) \mapsto x+y$ is a finite set in $\link_\gamma(X) \times -\link_\gamma(Y)$, each lying in the relative interior of a facet.

The formula on the right hand side coincides with the cup product formula for Chow cohomology of toric varieties given by the {\em fan displacement rule} of Fulton and Sturmfels \cite[Theorem~4.2]{FultonSturmfels}. 

\medskip

Let $T^k$ be the set of tropical cycles in $N_\RR$  of codimension $k$ with rational multiplicities.  Then $T^k$ is a $\QQ$-vector space where scalar multiplication acts on the multiplicities and the addition operation is defined as follows.  For $X, Y\in T^k$, their sum $X \cyclesum Y$ is obtained by taking the union $X \cup Y$ and adding multiplicities on the overlaps.  For a generic point $\omega \in X \cup Y$, we have $$\mult{\omega}{X \cyclesum Y} = \mult{\omega}{X} + \mult{\omega}{Y},$$ while the actual multiplicities depend on the generic choice of $\omega$.
 If $\mult{\omega}{X \cyclesum Y} = 0$, then we remove the facet containing $\omega$ from $X \cyclesum Y$.  If $\mult{\omega}{X \cyclesum Y} = 0$ for $\omega$ in a dense set of $X \cup Y$, then $X \cyclesum Y$ is the zero cycle.  Note that we only add tropical cycles of the same codimension and that the sum preserves codimension unless it is the zero cycle.  

\begin{remark}\label{rmk:decompose}
\begin{enumerate}
\item  Every tropical cycle can be decomposed as a linear combination of tropical cycles with positive multiplicities. This can be done, for example, by adding and subtracting tropical cycles that are affine spans of faces with negative multiplicities. By definition the stable intersection is seen to be distributive over sums of tropical cycles with positive multiplicities. For example $\link_\gamma(X_1+X_2)-\link_\gamma(Y)=(\link_\gamma(X_1)+\link_\gamma(X_2))-\link_\gamma(Y)=$
$(\link_\gamma(X_1)-\link_\gamma(Y))+(\link_\gamma(X_2)-\link_\gamma(Y))$
with multiplicity as full-dimensional cycles in $N_\RR$.
\item
A polyhedral complex in $\RR^n$ is called {\em locally balanced} if it is pure dimensional and the link of every codimension one face positively spans a linear subspace of $\RR^n$.  We can define stable intersections of locally balanced complexes without multiplicities using the same definition, and the set-theoretic parts of the following results still hold.
\end{enumerate}
\end{remark}

The following result follows from Definition~\ref{def:stableIntersection} and is useful for computing stable intersections.

\begin{lemma}
\label{lem:computingStableIntersections}
For tropical cycles $X$ and $Y$ in $N_\RR$ with positive multiplicities,
$$
\supp(X \stint Y) = \mathop{\bigcup_{\sigma_1 \in X, \sigma_2 \in Y}}_{\dim(\sigma_1 \minksum \sigma_2) = n} \sigma_1 \cap \sigma_2 .
$$
\end{lemma}

\begin{proof}
For any $\omega \in X \cap Y$, $\link_\omega X - {(\link_\omega Y)} = N_\RR$ if and only if there are $\sigma_1 \in X$ and $\sigma_2 \in Y$ containing $\omega$ such that $\dim(\sigma_1\minksum\sigma_2) = \dim(\sigma_1 - \sigma_2) = n$.
\end{proof}

In~\cite{RGST} stable intersections were defined by taking limits of perturbed intersections. We will now show that this is equivalent to Definition~\ref{def:stableIntersection}.

\begin{proposition}
\label{prop:limit} 
 Let $X$ and $Y$ be tropical cycles in $N_\RR$ with positive multiplicities. 
 Then for any generic $v \in N_\RR$, we have
 $$\supp(X \stint Y)  = \lim_{\varepsilon \rightarrow 0} \supp(X) \cap ( \supp(Y) + \varepsilon v). $$  In particular, the limit set does not depend on the choice of generic $v$.
\end{proposition}

\begin{proof}
Let $\omega \in X \cap Y$, and $v \in N_\RR$.  Then $\omega \in \lim_{\varepsilon \rightarrow 0} X \cap ( Y + \varepsilon v)$ if and only if for every $\delta > 0$ there is an $\varepsilon > 0$  such that $X \cap (Y + \varepsilon v)$ contains a point within distance $\delta$ from $\omega$.  This holds if and only if $\link_\omega(X) \cap (\link_\omega(Y) + v) \neq \emptyset$, or equivalently $v \in \link_\omega{X} - {(\link_\omega{Y})}$. 

When $X$ and $Y$ are balanced with positive multiplicities, the support of $\link_\omega{X} -  {(\link_\omega{Y})}$ is either all of $N_\RR$ or has positive codimension.
There are only finitely many such positive-codimensional complexes as $\omega$ varies, so for  generic $v$
$$v \in \link_\omega X - {(\link_\omega Y)} \iff  \link_\omega X - {(\link_\omega Y)} = N_\RR.$$
The result follows from the two equivalences above.
\end{proof}

For any pure dimensional tropical cycles $X$ and $Y$, we say that the intersection $X \cap Y$ is {\em transverse} if the set of points $\omega \in X \cap Y$ for which $\supp(\link_\omega X)$ and $\supp(\link_\omega Y)$ are both linear spaces that span $N_\RR$ together is dense in $X \cap Y$.  If $X$ and $Y$ intersect transversely, then it follows from Lemma~\ref{lem:computingStableIntersections} that $X \stint Y = X \cap Y$.  For any point $\omega$ in $X \stint Y$ lying in the relative interiors of some facets $\sigma \in X$ and $\sigma' \in Y$ for some choice of polyhedral structure on $X$ and $Y$, we have the following simplified multiplicity formula that follows from the definition of the stable intersection:
\begin{equation}
\label{eqn:transverseMult}
\mult{\omega}{X \stint Y} = \mult{\omega}{X} \mult{\omega}{Y} [N: N_\sigma + N_{\sigma'}].
\end{equation}

Let us fix generic $v \in N_\RR$ and $\varepsilon > 0$.
For any polyhedra $\sigma \in X$ and $\tau \in Y$, the intersection $\sigma \cap (\tau + \varepsilon v)$ is either empty or contains a point in the relative interior of both $\sigma$ and $\tau + \varepsilon v$.  In the latter case, $\codim(\sigma \cap (\tau+\varepsilon v)) = \codim(\sigma)+\codim(\tau)$. Hence the intersection $X \cap (Y+\varepsilon v)$ is transverse and has codimension equal to $\codim(X)+\codim(Y)$.

We can assign multiplicities to the limit in Proposition~\ref{prop:limit} as follows.  We say that a facet $\gamma \in X \stint Y$ is the {\em limit} of the facet $\sigma \cap (\tau+\varepsilon v)$ in $X \cap (Y + \varepsilon v)$ for sufficiently small $\varepsilon > 0$ if $\sigma, \tau \supset \gamma$ and $\link_\gamma(\sigma) \cap (\link_\gamma(\tau)+v) \neq \emptyset$.  Since for generic $v$ and $\varepsilon$ the intersection $X \cap (Y + \varepsilon v)$ is transverse, hence stable, the facet $\sigma \cap (\tau+\varepsilon v)$ has multiplicity given by formula~(\ref{eqn:transverseMult}) above.  
Combining with the multiplicity formula in Definition~\ref{def:stableIntersection}, we see that 
\begin{equation}
\label{eqn:perturbedMult}
\mult{\gamma}{X \stint Y} = \sum_{\sigma, \tau} \mult{\sigma \cap (\tau + \varepsilon v)}{X \stint (Y + \varepsilon v)}
\end{equation}
where $v$ is generic, $\varepsilon > 0$ is sufficiently small, and the sum is over facets $\sigma \in X$ and $\tau \in Y$ such that the limit of $\sigma \cap (\tau+\varepsilon v)$ is $\gamma$. 

We will now see that stable intersections and multiplicities behave well under taking links and quotienting out by lineality. In the following, for a rational linear space $L \subset N_\RR$, $N_\RR/L$ is equipped with the lattice $N/N_L$.

\begin{lemma}
\label{lem:linklineality}
Let $X$ and $Y$ be tropical cycles with positive multiplicities.
\begin{enumerate}
\item For $\omega \in X \stint Y$, we have $\link_\omega(X \stint Y) = \link_\omega(X) \stint \link_\omega(Y)$.
\item For a rational linear space $L$ contained in the lineality spaces of both $X$ and $Y$, we have $(X \stint Y)/L = (X / L) \stint (Y / L)$.
\end{enumerate}
\end{lemma}

\begin{proof}
The set-theoretic part of the first statement follows from Lemma~\ref{lem:computingStableIntersections} and the multiplicity statement follows from the fact that $\link_v (\link_\omega Z) = \link_{\omega+\varepsilon v} Z$ for any sufficiently small positive real number $\varepsilon$. Indeed, for generic $\gamma$,
$$\mult{v}{\link_\omega(X\stint Y)}=\mult{\omega+\varepsilon v}{X\stint Y}=\mult{\gamma}{\link_{\omega+\varepsilon v}(X)-\link_{\omega+\varepsilon v}(Y)}=$$
$$\mult{\gamma}{\link_v(\link_{\omega}(X))-\link_v(\link_{\omega}(Y))}=\mult{v}{\link{\omega}(X)\stint\link_\omega(Y)}.$$

For the second statement, first observe that for a unimodular coordinate change $U$ we have $U(X\stint Y)=U(X)\stint U(Y)$. Pick a lattice basis of $N_L$ and extend it to a lattice basis of $N$. In this basis $L$ is a coordinate subspace and its presence in $X$ and $Y$ does not affect the construction of the stable intersection other than having to take the product with $L$.
\end{proof}

\begin{lemma}
\label{lem:diagonal}
Let $X$ and $Y$ be tropical cycles in $N_\RR$ with positive multiplicities. If we identify $x\in N_\RR$ with $(x,x)\in N_\RR \times N_\RR$ then $X\stint Y=(X\times Y)\stint \{\Delta\}$,  where $\Delta = \{(x,x) : x \in N_\RR\}$ is the diagonal with multiplicity $1$ in $N_\RR\times N_\RR$ and is identified with $N_\RR$.
\end{lemma}
\begin{proof}
Let $\omega\in X\cap Y$, $\sigma_1\in\link_\omega(X)$, and $\sigma_2\in\link_\omega( Y)$.  Let $A_1$ and $A_2$ be matrices whose sets of columns are lattice bases for $N_{\sigma_1}$ and $N_{\sigma_2}$ respectively.  Then the columns of the matrix 
$
\left(
\begin{array}{ccc}
A_1 & 0 & I \\
0 & A_2 & I \\
\end{array}
\right)
$
span $N_{\sigma_1 \times \sigma_2} + N_\Delta$ over $\ZZ$, while the columns of $\left(
\begin{array}{cc}
A_1 & -A_2 \\
\end{array}
\right)
$
span $N_{\sigma_1 - \sigma_2}$ over $\ZZ$.   Hence $\dim((\sigma_1 \times \sigma_2)\minksum\Delta) = 2n$ if and only if $\dim(\sigma_1\minksum \sigma_2) = n$. 
Moreover, 
\begin{align*}[N \times N : N_{\sigma_1 \times \sigma_2}+N_\Delta] &= [N \times N : \{0\}\times (-N_{\sigma_1})+\{0\}\times N_{\sigma_2}+N_\Delta] \\ &= [N \times N : \{0\}\times (N_{\sigma_1}+N_{\sigma_2})+N_\Delta] \\ &= [N :N_{\sigma_1}+N_{\sigma_2}].\end{align*}  For (generic) $v_1, v_2 \in N_\RR$, we have $(v_1, v_2) \in (\sigma_1 \times \sigma_2)-\Delta$ if and only if $v_1 - v_2 \in \sigma_1 - \sigma_2$.  The assertion follows from Definition~\ref{def:stableIntersection}.
\end{proof}

The next three lemmas will be used to prove Theorems~\ref{thm:balanced} and~\ref{thm:assoc}.  Their proofs, which are elementary but difficult, are given in the Appendix.

\begin{lemma}
\label{lem:hyperplane}
Let  $X$ be a tropical cycle and $H$ be a tropical cycle whose support is an affine hyperplane, both with positive multiplicities.  Then $X \stint H$ is also a tropical cycle, possibly zero, with $\codim(X \stint H) = \codim(X)+1$.
\end{lemma}

\begin{lemma}\label{lem:iterate}
Let $X$ be an arbitrary tropical cycle with positive multiplicities.  Suppose $Y$ is a tropical cycle of codimension $r$ whose support is an affine linear space such that $Y = ((H_1 \stint H_2) \cdots H_r)$ where $H_1,\dots,H_r$ are tropical cycles with positive multiplicities whose supports are affine hyperplanes. Then
$$
X \stint Y = (((X \stint H_r)\stint H_{r-1})\cdots H_1).
$$
In particular, it follows that $X\stint Y$ is a tropical cycle since the right hand side is a tropical cycle by Lemma ~\ref{lem:hyperplane}.
\end{lemma}
\noindent The stable intersections on the right hand side are well defined by Lemma~\ref{lem:hyperplane}.

\begin{lemma}
\label{lem:assocLinear}
Let $X, L_1,$ and $L_2$ be tropical cycles with positive multiplicities, and suppose that the supports of $L_1$ and $L_2$ are affine linear spaces.  Then
$$
X \stint (L_1 \stint L_2) = (X \stint L_1) \stint L_2.
$$
\end{lemma}

\begin{theorem}
\label{thm:balanced} 
For tropical cycles $X$ and $ Y$, the stable intersection $X \stint  Y$ is also a tropical cycle, balanced with $$\codim(X \stint  Y) = \codim(X) + \codim( Y).$$  
\end{theorem}

\begin{proof}

Every tropical cycle can be decomposed as a cycle difference (union) of positive tropical cycles, and stable intersection is distributive over the cycle sum; see Remark~\ref{rmk:decompose}.  Hence we can assume that $X$ and $Y$ have positive multiplicities.
By Lemma~\ref{lem:diagonal} the stable intersection $X \stint  Y$ can be identified with $(X\times Y)\stint \{\Delta\}$ where $\Delta$ is the diagonal in $N_\RR \times N_\RR$. By Lemma~\ref{lem:iterate} taking the stable intersection with the diagonal $\Delta$ is a tropical cycle, balanced with expected codimension.  Then $\codim(X\times  Y \stint \{\Delta\}) = \codim(X)+\codim( Y)+n$ in $N_\RR \times N_\RR$.  Identifying $\Delta$ with $N_\RR$ reduces the codimension by $n$, so we have $\codim(X\stint  Y) = \codim(X)+\codim( Y)$.
\end{proof}

\begin{theorem}
\label{thm:assoc}
Stable intersection is associative, i.e.\ for any three tropical cycles $X$, $ Y$, and $ Z$, we have 
$$
\left( X \stint  Y \right) \stint  Z = X \stint \left(  Y \stint  Z \right).
$$
\end{theorem}

\begin{proof}
First note that for tropical cycles $A,B,C$ in $N_\RR$ we have
\begin{equation}
\label{eqn:A}
(A \stint B) \times C = (A \times C) \stint (B \times N_\RR)
\end{equation}
as cycles in $N_\RR \times N_\RR$.
This follows from Definition~\ref{def:stableIntersection}. The equation holds with multiplicities since the lattice indices of Definition~\ref{def:stableIntersection} stay the same when going to the bigger lattices.
By the same argument, for a lattice $M$, a cycle $A$ in $M_\RR\times N_\RR$, a cycle $B$ in $M_\RR$ and the projection $p:M_\RR\times N_\RR\rightarrow M_\RR$, we have
\begin{equation}
\label{eqn:B}
p(A)\stint B = p(A\stint (B \times N_\RR)).
\end{equation}

Let $\pi:N_\RR\times N_\RR\rightarrow N_\RR$ and $\pi':N_\RR\times N_\RR\times N_\RR\rightarrow N_\RR\times N_\RR$ be the projections to the first coordinates and the first and last coordinates, respectively. Let $\Delta = \{(x,x) : x \in N_\RR\} $, $\Delta' =\{(x,x,x) : x \in N_\RR\}$, $\Delta_{12} = \{(x,x,y):x,y\in N_\RR\}$, and $\Delta_{13} = \{(x,y,x):x,y \in N_\RR\}$. Using the diagonal trick from Lemma~\ref{lem:diagonal},
\begin{align*}
(X \stint  Y) \stint  Z 
& = \pi((\pi((X \times  Y) \stint \Delta) \times  Z) \stint \Delta) \\ 
& = \pi(\pi'(((X \times  Y) \stint \Delta) \times  Z) \stint \Delta) \\ 
& = \pi(\pi'((X \times  Y \times  Z) \stint \Delta_{12}) \stint \Delta) \\ 
& = \pi(\pi'(((X \times  Y \times  Z) \stint \Delta_{12}) \stint \Delta_{13} )) \\ & = (\pi\circ\pi')((X \times  Y \times  Z) \stint \Delta').
\end{align*}
The third equality follows from (\ref{eqn:A}), the fourth from (\ref{eqn:B}),
while the last follows from Lemma~\ref{lem:assocLinear} and the fact that $\Delta_{12} \stint \Delta_{13} =\Delta'$.
The formula on the right hand side is clearly associative, and so is the stable intersection.
\end{proof}

\begin{proposition}
\label{prop:AR}
Our definition coincides with the Allermann--Rau intersection product of tropical cycles.  
\end{proposition}

\begin{proof}
Using the diagonal trick from Lemma~\ref{lem:diagonal} and rewriting the diagonal as the stable intersection of hyperplanes, we can reduce to the case when one of the tropical cycles is a usual hyperplane $H$ with multiplicity $1$ and the other is a tropical cycle $X$.  Then both sets contain points in the support of $X$ whose link is not contained in $H$.  To compute multiplicities, suppose $H$ is defined by $\omega_1 = 0$.  By taking links if necessary, we may assume that $X \stint H$ is a linear space $L \subset H$, and that $X$ consists of $k$ cones, each of which is spanned by one of the vectors $r^{(1)}, \dots, r^{(k)}$ together with $L$.  Let $m_1, \dots, m_k$ be the multiplicities of those cones respectively.  Using the formula (\ref{eqn:transverseMult}) with perturbation $H + \varepsilon e_1$, we compute the multiplicity of $L$ in $X \stint H$ to be 
$$\sum_{i : r^{(i)}_1 > 0} m_i [N : N_H + N_{r^{(i)}+L}] = \sum_{i : r^{(i)}_1 > 0} m_i r^{(i)}_1.$$  This is easily seen to coincide with the Allermann--Rau definition of multiplicities using the tropical polynomial $\max(0,\omega_1)$ that defines $H$.
\end{proof}

Let $n$ be a positive integer.  For $k=0,1,\dots,n$, let $T^k$ be the $\QQ$-vector space of tropical codimension $k$ cycles in $N_\RR$ with rational multiplicities, where addition is the union.
Let $T$ be the direct sum $T=\oplus_{k=0}^n T^k$.  The stable intersection gives multiplication on $T$.  We have shown that the stable intersection is commutative, associative, and distributive over cycle-addition.

\begin{theorem}
The set $T$ of tropical cycles form a graded $\QQ$-algebra where addition is union and multiplication is stable intersection.
\end{theorem}

We will see in Section~\ref{sec:PolytopeAlgebra} that the algebra $T$ is isomorphic to the polytope algebra of McMullen.

\section{Stable intersections as tropical varieties of ideals}
\label{sec:ideals}

In this section, we interpret stable intersections of tropical varieties as tropicalizations of intersections after generic perturbations by rescaling, as we will make precise below.

Until now we have not assumed that our tropical cycles are fans, but for the rest of this paper they will be. The definition of tropical varieties below can be extended to the case where a valuation of the coefficient field is taken into account. See~\cite{MaclaganSturmfels} for details. In that setting tropical varieties need not be fans. For simplicity we consider only the fan case here.

Let $\kk$ be an algebraically closed field, and $R$ be the Laurent polynomial ring of the torus $N \otimes_{\ZZ} \kk^*$, i.e.\ $R = \kk[M]$ where $M = \Hom_\ZZ(N,\ZZ)$.  For $\omega\in N_\RR$ the \emph{initial form} $\init_\omega(f)$ of $f\in R\setminus\{0\}$ is the sum of terms $cx^v$ of $f$ with $\langle v,\omega\rangle$ maximal. The \emph{initial ideal} of $I\subseteq R$ is $\init_\omega(I):=\langle\textup{in}_\omega(f):f\in I\rangle$. The {\em tropical variety} of an ideal $I$ in $R$ is 
$$
\cT(I) = \{\omega \in N_\RR: \init_\omega(I) \neq R\}.
$$
It can be given a fan structure as a subfan of the Gr\"obner fan of $I$ (after possibly a homogenization of $I$), and it is a tropical cycle with multiplicities given by 
$$\mult{\omega}{\cT(I)}:=\dim_\kk(\kk[M\cap C^\perp]/\langle\init_\omega(I)\rangle)$$
for generic points $\omega$ in a facet $C$. Equivalently the multiplicity can be defined as
$$
\mult{\omega}{\cT(I)}:= \sum_{P \in \Ass(\init_\omega(I))}\operatorname{mult}(P, \init_\omega(I)).
$$
See \cite{BJSST, MaclaganSturmfels} for more details.  If $L$ is a subspace of the lineality space of $\cT(I)$, then $\cT(I)/L$ is a tropical cycle in $(N/N_L)_\RR$ with inherited multiplicities from $\cT(I)$, and 
\begin{equation}
\label{eqn:modOutAlgebra}
\cT(I)/L = \cT(I \cap \kk[L^\perp\ \cap M])
\end{equation}
where the lattice $L^\perp \cap M$ is naturally identified with $\Hom_\ZZ(N/N_L, \ZZ)$.

\begin{lemma}
\label{lem:initialsum}
Let $I\subseteq \kk[x_1,\dots,x_n]$ be an ideal with $n\geq 1$. Let $\omega\in\RR^n$ have $\omega_1=0$. Considered as ideals in $\kk(\alpha)[x_1,\dots,x_n]$ we have
$$\textup{in}_{\omega}(\langle I\rangle +\langle x_1-\alpha \rangle)=\textup{in}_{\omega}(\langle I\rangle)+\textup{in}_{\omega}(\langle x_1-\alpha \rangle).$$
\end{lemma}
\begin{proof}
The inclusion $\supseteq$ is clear and we will prove $\subseteq$. The
left hand side is generated by elements of the form
$\textup{in}_{\omega}({1\over p}(f+g\cdot(x_1-\alpha)))$ with $f\in
\langle I\rangle\cap \kk[\alpha][x_1,\dots,x_n]$,
$g\in{\kk[\alpha][x_1,\dots,x_n]}$, $p\in \kk[\alpha]$. Since $p$ is a unit, we may
ignore the $1/p$ factor in our argument.

We argue that without loss of generality $f\in I$. That is, $f$ does not involve~$\alpha$. If $f\not\in I$, then consider the expression $f=\sum_i c_iF_i$ with $c_i\in \kk[\alpha][x_1,\dots,x_n]$ and $F_i\in I$ and perform polynomial division of $c_i$ modulo $\alpha-x_1$ to obtain $c_i=g_i'(x_1-\alpha)+r_i$ for some $g_i'\in \kk[\alpha][x_1,\dots,x_n]$ and $r_i\in \kk[x_1,\dots,x_n]$. We now have $f+g(x_1-\alpha)=\sum_i c_iF_i+g(x_1-\alpha)=\sum_i (g_i'(x_1-\alpha)+r_i)F_i+g(x_1-\alpha)=\sum_i r_iF_i +((\sum_i g_i'F_i)+g)(x_1-\alpha)$. Here $\sum_i r_iF_i$ indeed is in $I$. 

Consider the degree of
$f,g\cdot(x_1-\alpha)$ and $f+g\cdot(x_1-\alpha)$ in the
$\omega$ grading. Since $\textup{in}_{\omega}(g\cdot(x_1-\alpha))$ contains
$\alpha$, its terms of some $\omega$-degree cannot cancel completely
with terms of $f$, if $g\not=0$. Therefore,
$\textup{in}_{\omega}(f+g\cdot(x_1-\alpha))$ is either
$\textup{in}_{\omega}(f)$, if the degree of $f$ is highest, or
$\textup{in}_{\omega}(g\cdot(x_1-\alpha))$ if the degree of $g\cdot(x_1-\alpha)$
is highest, or
$\textup{in}_{\omega}(f)+\textup{in}_{\omega}(g\cdot(x_1-\alpha))$ if
the degrees are equal.
\end{proof}

The \emph{saturation} of an ideal $I$ in a ring $R$ by an element $f\in R$ is the ideal $(I:f^\infty):=\{g\in R : gf^m\in I \text{ for some } m\in \NN\}$.
\begin{lemma}
\label{lem:projElim}
Let $I$ be an ideal in $\kk[x_1,\dots,x_n]$.  Then
$$
\cT(I) \nsubseteq \{\omega: \omega_1 = 0\}  \iff (I : (x_1 x_2 \cdots x_n)^\infty) \cap \kk[x_1] = \{0\}.
$$
\end{lemma}

\begin{proof}
Consider the projection $\pi_1$ of $(\kk^*)^n$ onto the $x_1$ axis. 
The variety defined by the ideal $J = (I : (x_1 x_2 \cdots x_n)^\infty) \cap \kk[x_1]$ is the Zariski closure $\overline{\pi_1(V(I))}$.  By the fundamental theorem of tropical geometry, the tropical variety of $\overline{\pi_1(V(I))}$ is the image of $\cT(I)$ under the projection onto the first coordinate axis~\cite[Theorem~3.2.3]{MaclaganSturmfels}.
There are three possibilities for $\overline{\pi_1(V(I))}$:
 
\begin{itemize}
\item Empty: LHS is false, and $J = \kk[x_1] \neq \{0\}$.
\item Finitely many points: the tropicalization of the projection is $\{0\}$, so the LHS is true, and $J$ contains a nonzero polynomial in $x_1$ vanishing on those points.
\item All of $\kk^*$: the LHS is true, and $J = \{0\}$.\qedhere
\end{itemize}

\end{proof}

Let $I$ be an ideal in $\kk[x_1,\dots,x_n]$.  For all but finitely many $c \in \kk$, we have
$$
\cT(I + \langle x_1 - c \rangle)= \cT(\langle I \rangle + \langle x_1 - \alpha \rangle)
$$
where the ideal in the right hand side is in $\kk(\alpha)[x_1,\dots,x_n]$. Indeed for a homogeneous ideal $I$, to find the right hand side a finite number of Gr\"obner basis computations are required, see~\cite{BJSST}. 
Except for a finite number of choices of $c$, 
the Gr\"obner bases of $ \langle I \rangle + \langle x_1 - \alpha \rangle$ are also Gr\"obner bases for $I + \langle x_1 - c \rangle$ after substituting $\alpha$ with $c$.  


\begin{lemma}
\label{lem:slicingbinomial}
Let $I$ be an ideal in $\kk[x_1,\dots,x_n]$, and $H$ be the tropical cycle in $N_\RR$ defined by $\omega_1=0$ with multiplicity one at all points.  Then
$$
\cT(I) \stint H = \cT(\langle I \rangle + \langle x_1 - \alpha \rangle)
$$
as tropical cycles, where the ideals on the right hand side are in $\kk(\alpha)[x_1,\dots,x_n]$ and $\alpha$ is transcendental over $\kk$.
\end{lemma}

\begin{proof}
Let $\omega \in \RR^n$ such that $\omega_1 = 0$.  Then the following statements are equivalent:
\begin{enumerate}
\item $\omega\in \cT(I) \stint H$
\item $\cT(\init_\omega(I))\nsubseteq H$
\item $(\init_\omega(I): (x_1 x_2 \cdots x_n)^\infty )\cap\kk[x_1]=\{0\}$
\item $\init_\omega(\langle I\rangle)+\langle x_1-\alpha\rangle$ is monomial free
\item $\omega\in\cT(\langle I \rangle + \langle x_1 - \alpha \rangle)$.
\end{enumerate}
The equivalence $(1)\Leftrightarrow (2)$ follow from the definitions of tropical varieties and stable intersection and the statement $\link_\omega(\cT(I)) = \cT(\init_\omega I)$. 
The equivalences $(4)\Leftrightarrow (5)$ and  $(2) \Leftrightarrow (3)$  follow from Lemmas \ref{lem:initialsum} and \ref{lem:projElim} respectively.  To see $(4) \Rightarrow (3)$, note that an ideal containing $x_1 - \alpha$ and a non-zero polynomial in $\kk[x_1]$ must be the unit ideal.  


To see $(3) \Rightarrow (4)$, suppose $(4)$ does not hold, then there is a monomial $m$ such that
$m = \sum_ip_if_i + g \cdot (x_1 - \alpha)$ where $f_i\in \init_\omega(I)$, $p_i\in\kk(\alpha)$ and $g\in \kk(\alpha)[x_1,\dots,x_n]$.  After clearing  denominators, we may assume that all $p_i\in\kk[\alpha]$, $g\in \kk[\alpha,x_1,\dots,x_n]$ and $m=q\cdot x^v$ with $q\in\kk[\alpha]\setminus\{0\}$ and $v\in \NN^n$.  By substituting $\alpha$ with $x_1$ we get $q\cdot x^v = \sum_ip_if_i$ with $p_i$ and $q$ being in $\kk[x_1]$.  Hence $q\cdot x^v$ is in $\init_\omega(I)$ and $q$ is a non-zero element in $(\init_\omega(I):x_1\cdots x_n^\infty)\cap\kk[x_1]$.

Now we need to show that multiplicities coincide on a dense open subset.  By taking links and quotienting out lineality space as in Lemma~\ref{lem:linklineality} and equation (\ref{eqn:modOutAlgebra}) we can reduce to the case where $I$ is one dimensional and the intersection is just $\{0\}$.  In this case, using Definition~\ref{def:stableIntersection} of stable intersection multiplicities, we have  $$\mult{0}{\cT(I)\stint H} = \sum_{\sigma} \mult{\sigma}{\cT(I)} [N : N_\sigma + N_H]$$ where $\sigma$ runs over rays of $\cT(I)$ such that $\sigma + H$ contains a fixed generic element $(a_1,\dots,a_n)\in N_\RR$.  For a ray $\sigma$ of $\cT(I)$, let $v_\sigma$ denote the generator of $N_\sigma$.  Then 
$$
\mult{0}{\cT(I)\stint H} = \sum_{\sigma} \mult{\sigma}{\cT(I)} {v_\sigma}_1
$$
where $\sigma$ runs over all rays of $\cT(I)$ such that the first coordinates ${v_\sigma}_1$ and $a_1$ have the same sign.  The right hand side is equal to the degree of the projection of the curve $V(I)$ onto the first coordinate, by the Sturmfels--Tevelev formula for push-forward of multiplicities \cite{SturmfelsTevelev}.  See the paragraph above Example~\ref{ex:ST}.  This degree is the degree of $I + \langle x - \alpha\rangle$ for a generic $\alpha$, which is the multiplicity of the origin of $\cT(I + \langle x - \alpha \rangle)$.
\end{proof}

Let $I$ and $J$ be ideals in $\kk[x_1,\dots,x_n]$.  Let $\KK = \kk(c_1,c_2,\dots,c_n)$ be the field of rational functions in indeterminates $c_1, c_2, \dots, c_n$.
Define $I'$ to be the ideal in $\KK[x_1,\dots,x_n]$ generated by $I$ and $J'$ to be the ideal generated by the image of $J$ in $\KK[x_1,\dots,x_n]$ under the ring homomorphism $\kk[x_1,\dots,x_n] \rightarrow \KK[x_1,\dots,x_n]$ given by $x_i \mapsto c_i x_i$.

\begin{lemma}
\label{lem:gfanmultiplied}
  The change of coordinates from $J$ to $J'$ preserves the Gr\"obner fan and the tropical variety, i.e.\ 
$\Gfan(J) = \Gfan(J')$ and $\cT(J) = \cT(J')$.
\end{lemma}
The fans on the left hand sides are defined with respect to $\kk$ while those on the right are with respect to $\KK$.

\begin{proof}
Let $\prec$ be any term order.  Buchberger's S-pair algorithm for Gr\"obner bases commutes with the coordinate change $x_i \mapsto c_i x_i$, so for a Gr\"obner basis  $G$ of $J$ with respect to $\prec$, the image of $G$ under the map $x_i \mapsto c_i x_i$ forms a Gr\"obner basis of $J'$ with respect to $\prec$.
\end{proof}



We can now prove the main theorem of this section.  A similar result can be found in the work of Osserman and Payne \cite[Proposition 2.7.8]{OssermanPayne}.

\begin{theorem}\label{thm:genericvariables} With the notation above, we have
$$\cT(I')\stint\cT(J')=\cT(I'+J').$$
\end{theorem}

\begin{proof}
Using Lemmas~\ref{lem:gfanmultiplied}, \ref{lem:diagonal}, \ref{lem:slicingbinomial} and associativity of stable intersection we get $\cT(I')\stint\cT(J')=\cT(I)\stint\cT(J)=(\cT(I)\times\cT(J))\stint\Delta=\cT(I'' + J'' +\langle c_1x_1-X_1,\dots,c_nx_n-X_n\rangle) = \cT(I'+J')$ where $I''$ is the ideal generated by $I$ in $\kk[x_1,\dots,x_n,X_1,\dots,X_n]$ and $J''$ the ideal generated by $J$ after substituting $X_i$ for $x_i$. The second and fourth equalities are after identification of $N_\RR$ with the diagonal in $N_\RR\times N_\RR$.
\end{proof}

For some problems such as computation of resultants studied in \cite{tropRes}, we want to compute the Newton polytope of a polynomial after substituting one or more variables with generic constants.  This amounts to projecting the Newton polytope onto the coordinate subspace of the remaining variables.  By the following theorem, we can perform this operation on tropical hypersurfaces, as orthogonal projection of a polytope onto a linear space is equivalent to stably intersecting the tropical hypersurface of the polytope with the linear space.  In order to work with rational polytopes that are not integral, we allow multiplicities to be rational.
The tropical hypersurface $\cT(P)$ of a rational polytope $P\subseteq\QQ^n$ is then defined to be the set of normal cones of $P$ of dimension at most $n-1$. 
The multiplicity assigned to a normal cone $C$ of an edge $e$ of $P$ is the \emph{lattice length} of $e$ i.e. the rational number $\|e\|/\|p\|$ where $p$ is a generator for $C^\perp\cap \ZZ^n$.
The tropical hypersurface $\cT(P)$ is balanced around every ridge $R$ because the oriented edges of the two dimensional face of $P$ dual to $R$ sum to zero.

\begin{theorem}
Let $P\subset \QQ^n$ be a polytope, $L \subset \QQ^n$ be a rational linear subspace, and $\pi : \QQ^n \rightarrow L$ be the orthogonal projection.  Then 
$$\cT(\pi(P))=(\cT(P)\stint L) \minksum L^\perp.$$
\end{theorem}

\begin{proof}
Since $L^{\perp}$ is contained in the lineality space of both sides of the equation, it suffices to show that
$$
\cT(\pi(P)) \cap L = \cT(P) \stint L.
$$
For $\omega \in \QQ^n$, let $P^\omega$ denote the face of $P$ supported by the hyperplane whose normal vector pointing away from $P$ is $\omega$.  Then $\link_\omega(\cT(P)) = \cT(P^\omega)$.    For $\omega \in L$, we also have $(\pi(P))^\omega = \pi(P^\omega)$, and
\begin{align*}
\omega \in \cT(P) \stint L &\iff \link_\omega \cT(P) \minksum L = \QQ^n \\
&\iff \cT(P^\omega) \minksum L = \QQ^n \\
&\iff \cT(P^\omega) \textup{ contains a cone $C$ such that  } \textup{dim}(L+C)=n \\
&\iff P^\omega \textup{ contains an edge not in } L^\perp\\
& \iff (\pi(P))^\omega \text{ contains at least two distinct points}\\
& \iff \omega \in \cT(\pi(P)).
\end{align*}
\end{proof}
The normal fan of the projection of a polytope onto a linear subspace is the restriction of the normal fan to the linear subspace \cite[Chapter 7]{Zie}.

\section{Volume, mixed volume, and Bernstein's Theorem}
\label{sec:Bernstein}

In this section we study stable intersections of hypersurfaces of polytopes. 
For polytopes $P_1, P_2, \dots, P_n$ in $\QQ^n$, the stable intersection of their hypersurfaces  $\cT(P_1) \stint \cdots \stint \cT(P_n)$ is either empty or consists only of the origin.  



\begin{lemma}
\label{lem:polytopesum}
Let $P$ and $Q$ be polytopes such that $P \cup Q$ is also a polytope.  Then 
$(P\cup Q) \minksum (P \cap Q) = P \minksum Q$,
where $\minksum$ denotes the Minkowski sum.
\end{lemma}

\begin{proof}
For a closed convex body $K \subset \QQ^n$, let $K(x)$ denote its support function, i.e.\ for any $x \in \QQ^n$, $K(x) = \max\{x\cdot v: v \in K\}$. We claim that 
$$\{x\cdot v:v\in P\cap Q\}=\{x\cdot v:v\in P\}\cap\{x\cdot v:v\in Q\}$$
An element on the left hand side is clearly on the right. Conversely, suppose $v'\in P$ and $v''\in Q$ such that $x\cdot v'=x\cdot v''$ is an element on the right hand side. Then $\textup{conv}(\{v',v''\})\subseteq \textup{conv}(P\cup Q)=P\cup Q$, implying that the closed line segments $\textup{conv}(\{v',v''\})\cap P$ and $\textup{conv}(\{v',v''\})\cap Q$ cover $\textup{conv}(\{v',v''\})$. Hence the two line segments have a common point $v\in \textup{conv}(\{v',v''\})\cap P\cap Q$. Because $v\in P\cap Q$ and $x\cdot v=x\cdot v'=x\cdot v''$, the value $x\cdot v'$ is also in the left hand side.

Taking maximum on both sides, using that the right hand side is a non-empty intersection of closed intervals in $\RR$, we get that $(P \cap Q)(x)=\min(P(x),Q(x))$. Therefore \begin{align*}((P\cup Q) \minksum (P \cap Q))(x) & = (P\cup Q)(x) + (P \cap Q)(x) \\ & = \max(P(x),Q(x)) + \min(P(x),Q(x)) \\ & = P(x)+Q(x) \\ &= (P\minksum Q)(x). \end{align*} Since closed convex bodies are uniquely determined by their support functions, we get $(P \cup Q) \minksum (P \cap Q) = P \minksum Q$.  
\end{proof}

For a polytope $P$ and a non-negative integer $r$, let $\cT^r(P)$ denote the stable intersection of $\cT(P)$ with itself $r$ times. In particular, $\cT^0(P)=\RR^n$ and $\cT^1(P)=\cT(P)$.  Recall that $\cyclesum$ denotes the sum of tropical cycles of the same dimension, which is taking the union of supports and adding the multiplicities on the overlap.

\begin{proposition}
\label{prop:freshmansDream}
Let $P$ and $Q$ be polytopes in $\QQ^n$ such that $P \cup Q$ is also a polytope.  Then 
\begin{enumerate}
\item $\cT(P\cup Q) \stint \cT(P \cap Q) = \cT(P) \stint \cT(Q)$ 
\item  $\cT^k(P\cup Q) \cyclesum \cT^k(P \cap Q) = \cT^k(P) \cyclesum \cT^k(Q)$ for any $k \geq 0$
\end{enumerate}
\end{proposition}

In the language of polytope algebra \cite{polytopeAlgebra}, the second part says that $\cT^k$ is a {\em valuation} for every $k \geq 0$.

\begin{proof} Let $R = P\cup Q$ and $S = P \cap Q$.  Since codimension adds under stable intersection, $\cT(P) \stint \cT(Q)$ is a codimension~$2$ tropical cycle. It is a subfan of $\cT(P) \cap \cT(Q)$, which has a fan structure derived from the normal fan of the polytope $P+Q$.  Similarly, $\cT(R) \stint \cT(S)$ is a codimension~$2$ subfan of the normal fan of $R+S$.  We also have $P+Q = R+S$ by Lemma~\ref{lem:polytopesum}.

For a codimension-2 cone $\sigma$ in the normal fan of $P+Q$,  $\dim(P^\sigma \minksum Q^\sigma)=2$.  In addition, $\sigma$ is in $\cT(P)\stint \cT(Q)$ if and only if $\dim(P^\sigma) \geq 1$ and $\dim(Q^\sigma) \geq 1$; similarly for $\cT(R)\stint \cT(S)$.  If $P(\omega) \neq Q(\omega)$ for $\omega$ in the relative interior of $\sigma$, then $\{P^\sigma , Q^\sigma\} = \{R^\sigma, S^\sigma\}$.  If $P(\omega) = Q(\omega)$, then $R^\sigma = P^\sigma \cup Q^\sigma$ and $S^\sigma = P^\sigma \cap Q^\sigma$.  In any case, because $\dim(R^\sigma)\leq\textup{max}(\dim(P^\sigma),\dim(Q^\sigma))$, it follows that $\sigma \in \cT(P)\stint \cT(Q)$ if and only if $\sigma \in \cT(R)\stint \cT(S)$.

To compare multiplicities, by taking the link at $\sigma$ and quotienting out by the lineality space, by Lemma~\ref{lem:linklineality}, we reduce the problem to the case when $P,Q,R,S$ lie in a two-dimensional plane.  By equation~\ref{eqn:perturbedMult}, the multiplicity at $0$ of two tropical hypersurfaces $\cT(P)$ and $\cT(Q)$ in the plane can be computed by translating one of them generically and adding up the intersection multiplicities.  By examining the dual (mixed) subdivision, we have
\begin{align*}
\mult{\sigma}{\cT(P) \stint \cT(Q)} & = \vol(P^\sigma \minksum Q^\sigma) - \vol(P^\sigma) - \vol(Q^\sigma), \text{ and}\\
\mult{\sigma}{\cT(R) \stint \cT(S)} & = \vol(R^\sigma \minksum S^\sigma) - \vol(R^\sigma) - \vol(S^\sigma).
\end{align*}
See \cite[Figures 3.1 and 9.7]{Sturmfels02} for pictures of tropical plane curves and the dual subdivisions.
To see that the two quantities on the right are equal we apply Lemma~\ref{lem:polytopesum} and the observations of the previous paragraph to get $\vol(P^\sigma \minksum Q^\sigma)=\vol(R^\sigma \minksum S^\sigma)$. Furthermore, $\vol(P^\sigma) + \vol(Q^\sigma)=\vol(R^\sigma) + \vol(S^\sigma)$ either because $\{P^\sigma , Q^\sigma\} = \{R^\sigma, S^\sigma\}$ or because $\vol(P^\sigma) + \vol(Q^\sigma)=\vol(P^\sigma\cup Q^\sigma) + \vol(P^\sigma\cap Q^\sigma)$. We conclude that multiplicities are equal.
\smallskip

To prove the statement (2), we proceed by induction on $k$. Since $\cT^0(P) = \RR^n$ for all $P$, the assertion is true for $k=0$. For $k=1$, it follows from Lemma~\ref{lem:polytopesum} and the fact that $\cT(P \minksum Q) = \cT(P) \cyclesum \cT(Q)$ for any polytopes $P$ and $Q$. Suppose $k \geq 2$.  Let $p = \cT(P)$, $q=\cT(Q)$, $r = \cT(P\cup Q)$, and $s = \cT(P\cap Q)$.  We have $r^{k-2} \cyclesum s^{k-2} = p^{k-2} \cyclesum q^{k-2}$ by the inductive hypothesis, and $rs = pq$ from part $(1)$.   Multiplying them gives $r^{k-1} s \cyclesum  r s^{k-1} = p^{k-1} q \cyclesum p q^{k-1}$.  Combining this with $(r\cyclesum s)(r^{k-1}\cyclesum s^{k-1}) = (p\cyclesum q)(p^{k-1}\cyclesum q^{k-1})$, which follows from the inductive hypothesis for $1$ and $k-1$, gives the desired identity $r^k \cyclesum s^k = p^k \cyclesum q^k$.
\end{proof}

\begin{theorem}
\label{thm:volume}
Let $P$ be a rational polytope in $\QQ^n$ and $\cT(P)$ be its tropical hypersurface.  Then 
$$
\vol(P) = \mult{0}{\cT^{\dim P}(P)}
$$
where $\vol$ denotes the $\dim(P)$-dimensional volume normalized to the integer lattice parallel to the affine span of $P$.
\end{theorem}

\begin{proof}
In \cite{Tverberg}, Tverberg showed that every polytope can be decomposed into simplices by a finite number hyperplane cuts.  Combining this with Proposition~\ref{prop:freshmansDream}(2) applied to the case when $P \cap Q$ is lower dimensional, we reduce the problem to the case when $P$ is a simplex. 

Now we will show the result for the case where $P$ is a simplex, by induction on the dimension of $P$.   When $P$ is one-dimensional, it is a line segment, and the assertion is true, as the multiplicity equals the lattice length of the segment.  Suppose $P$ is a $d$-dimensional simplex.  By quotienting out by the lineality space of $\cT(P)$ if necessary, we may assume that $P$ is full-dimensional in its ambient space.  By the inductive hypothesis, $\cT^{d-1}(P)$ is a one dimensional fan whose rays are facet normals of $P$ with multiplicities equal to the respective $d-1$ dimensional volumes of the facets.

 The multiplicity of the origin in $\cT(P) \stint \cT^{d-1}(P)$ is, by definition, the multiplicity of the Minkowski sum $\cT(P)- {\cT^{d-1}(P)}$ which is equal to $\RR^n$ as a set.  Each connected component $C$ of the complement of $\cT(P)$ is a simplicial full-dimensional cone, containing exactly one ray $R$ of $\refl{\cT^{d-1}(P)}$ in its interior because $\cT^{d-1}(P)$ has exactly $d+1$ rays, $d$ of which are rays of $\overline{C}$ and the remaining ray is the negative of a positive linear combination of the other $d$ rays.  Taking the Minkowski sum of $R$ with each facet of $C$, we get a triangulation of $C$.  Doing so for each complement component, we get a triangulation of $\RR^n$ (which is the normal fan of $P - {P}$).  These cones are precisely the full-dimensional cones of the form $\sigma \minksum R$ where $\sigma$ is a facet of $\cT(P)$ and $R$ is a ray of the tropical curve $\refl{\cT^{d-1}(P)}$, and they have disjoint interiors.  Hence, to compute the multiplicity of $\cT(P) - {\cT^{d-1}(P)}$, we need only to consider one such cone.  

Suppose $0, v_1, \dots, v_n$ are vertices of the simplex $P$. Let $r$ be the primitive vector perpendicular to the facet containing $0, v_1,\dots,v_{n-1}$ and $a$ be the normalized volume of that same facet. Let $\sigma$ be the maximal cone of $\cT(P)$ normal to the edge $\{0, v_n\}$ with multiplicity equal to the lattice length $l$ of the edge and $R$ be the cone spanned by $r$ with multiplicity $a$. Let $u_1,\dots,u_{n-1}$ be a lattice basis of $v_n^\perp\cap\ZZ^n$.
 According to Definition~\ref{def:stableIntersection} the multiplicity of the stable intersection of $\cT(P)$ and $\cT^{d-1}(P)$ is $l \cdot a \cdot |\det[r|u_1|\cdots|u_{n-1}]|$, which is equal to $a\, |r^T v_n|$.  This is equal to the volume $|\det[v_1|v_2|\cdots|v_n]|$ of $P$.
\end{proof}

\begin{corollary}
\label{cor:volumePolynomial}
The function $\vol_n(\lambda_1 P_1 \minksum \lambda_2 P_2 \minksum \cdots \minksum \lambda_n P_n)$ is a degree $n$ homogeneous polynomial in $\lambda_1, \lambda_2, \dots, \lambda_n$, and the coefficient of $\Pi_{i=1}^n \lambda_i^{a_i}$ is $$\frac{n!}{a_1! a_2! \cdots a_n!} \mult{0}{\Pi_{i=1}^n\cT^{a_i}(P_i)}.$$
\end{corollary}

\begin{proof}
We expand $\vol_n(\lambda_1 P_1 \minksum \lambda_2 P_2 \minksum \cdots \minksum \lambda_n P_n)$ using Theorem~\ref{thm:volume}.
The result now follows from the distributivity of stable intersection over cycle addition (taking unions) and the additivity of multiplicities in cycle sums.
\end{proof}

\begin{corollary}
For rational polytopes $P_1,P_2,\dots,P_n$ in $\QQ^n$, we have
$$
\mult{0}{\cT(P_1) \stint \cT(P_2) \cdots \cT(P_n)} = \MV(P_1, P_2, \dots, P_n)
$$
where $\MV$ denotes the mixed volume.
\end{corollary}

\begin{proof} The mixed volume $\MV(P_1,\dots,P_n)$ is the coefficient of $\lambda_1\lambda_2\cdots\lambda_n$ in the polynomial $\frac{1}{n!}\vol_n(\lambda_1 P_1 \minksum \lambda_2 P_2 \minksum \cdots \minksum \lambda_n P_n)$, which is equal to $$\mult{0}{\cT(P_1)\stint\cT(P_2)\cdots\cT(P_n)}$$ by Corollary~\ref{cor:volumePolynomial}.

Alternatively, we can see that the function
$$
(P_1, P_2, \dots, P_n) \mapsto \mult{0}{\cT(P_1) \stint \cT(P_2) \cdots \cT(P_n)}
$$
satisfies the axioms of the mixed volume, i.e.\ it is symmetric, multilinear, and $\mult{0}{\cT^{n}(P)} = \vol_n(P)$.
\end{proof}

We get a proof of Bernstein's Theorem.

\begin{theorem}[Bernstein]
Let $f_1, f_2, \dots, f_n $ be generic Laurent polynomials in $\CC[x_1^{\pm 1}, \dots, x_n^{\pm 1}]$, and let $I$ be the ideal generated by them.  If $I$ is zero-dimensional, then it has length equal to the mixed volume of the Newton polytopes of $f_1, f_2, \dots, f_n$.
\end{theorem}

\begin{proof}
If the coefficients are generic, then by Theorem~\ref{thm:genericvariables} the tropical variety of $I$ is the stable intersection of the tropical hypersurfaces of $f_1, f_2, \dots, f_n$, which is either empty or consists only of the origin with multiplicity equal to the mixed volume.  The length of the ideal is the multiplicity of the origin, by the definition of multiplicity.
\end{proof}

\section{Relation to Polytope Algebra}
\label{sec:PolytopeAlgebra}

For $r = 0, 1, \dots, n$, let $T^r$ be the vector space over $\QQ$ of rational tropical cycles of codimension $r$ in $\RR^n$.  Scalar multiplication acts on the multiplicities, and addition is taking union.  Then $T = T^0 \oplus T^1 \oplus \dots T^n$ is a graded algebra with stable intersection as multiplication.  

Let $\Pi$ be the polytope algebra of $\QQ^n$~\cite{polytopeAlgebra} defined as follows.  For a polytope $P \subset \QQ^n$, let $[P]$ denote the equivalence class of $P$ under the equivalence relation $P \sim P+v$ for $v \in \QQ^n$.  Then $\Pi$ consists of formal $\QQ$-linear combinations of $\{[P]: P \text{ is a polytope in } \QQ^n\}$, modulo relations $$[P \cup Q] + [P \cap Q] = [P] + [Q]$$ whenever $P \cup Q$ is a polytope.  The multiplication is Minkowski sum:
$$
[P] \cdot [Q] = [P\minksum Q].
$$
  The additive identity is $0 = [\emptyset]$, the class of the empty polytope, while the multiplicative identity is $1 = [0]$, the class of a point.  In $\Pi$, $([P]-1)^{n+1} = 0$ for every $n$-dimensional polytope $P$.  Hence the logarithm $$\log([P]) = \sum_{k\geq 1} (-1)^{k-1}([P]-1)^k/k$$ is defined for every non-empty polytope $P$.  Its inverse, the exponential map $\exp(z) = \sum_{k \geq 0} z^k/k!$, is defined for nilpotent elements $z$ in $\Pi$~\cite{polytopeAlgebra}.  McMullen showed that the polytope algebra $\Pi$ is in fact a graded algebra, where the $r$-th graded piece is linearly spanned by elements of the form $(\log([P]))^r$ for non-empty polytopes $P$~\cite[Lemma~20]{polytopeAlgebra}.

\begin{theorem}
There is an isomorphism of graded algebras
$$
\phi : \Pi \rightarrow T
$$
given by $\phi([P]) = 1 \oplus \cT(P) \oplus \frac{1}{2!} \cT^2(P) \oplus \cdots \oplus  \frac{1}{n!} \cT^n(P)$
for polytopes $P$ and linearly extending to $\Pi$.  
\end{theorem}

Under this map, $\log([P]) \mapsto \cT(P)$, so tropicalization is the logarithm.

\begin{proof}
This map is a well-defined homomorphism by Proposition~\ref{prop:freshmansDream}.
 To show that it is bijective, we use~\cite[Theorem 5.1]{FultonSturmfels}, which relies heavily on~\cite[Theorem 5.1]{McMullen_Simple}.

Any fan, after some subdivision, is a subfan of the normal fan of a polytope.  This can be achieved, for example, by taking the arrangement of hyperplanes spanned by any collection of cones in the original fan.  Any hyperplane arrangement gives rise to the normal fan of a zonotope.  
When restricted to this polytopal fan, the algebra of tropical cycles is isomorphic to the algebra of Minkowski weights defined by Fulton and Sturmfels, which was shown to be isomorphic to the polytope algebra  \cite{FultonSturmfels}.
\end{proof}

\begin{corollary}
Every tropical cycle is a rational linear combination of pure powers of
tropical hypersurfaces.
\end{corollary}

\begin{proof}
The $k$-th graded piece of $\Pi$ is spanned by $\{\log([P])^k : P \text{ is a polytope}\}$ as shown in \cite{polytopeAlgebra}.  For a polytope $P$, $\phi(\log([P])^k) = \cT^k(P)$, and the assertion follows because $\phi$ is an isomorphism.
\end{proof}

This corollary can be made constructive.  Given a tropical cycle, first make it a subfan of the normal fan of a simple polytope, for example, by extending it to a hyperplane arrangement and perturbing the facets of the dual zonotope.  We can then find a linear basis $F_1, \dots, F_m$ of of $T^1$ restricted on this fan, by linear programming, so that each $F_i$ has non-negative weights and hence are tropical hypersurfaces of polytopes.  Then for every $1 \leq k \leq n$, the $k$-fold products $F_{i_1}\stint F_{i_2}\stint\cdots\stint F_{i_k}$ linearly span $T^k$ restricted to this fan.  We can then decompose the input fan as a linear combination of these stable intersections of tropical hypersurfaces.

\medskip

By Theorem~\ref{thm:genericvariables}, pure powers of tropical
hypersurfaces are realizable, so we have the following.
\begin{corollary} Every tropical cycle is a rational linear combination of realizable tropical varieties.
\end{corollary}
It is not true, however, that every tropical cycle with {\em positive} weights is a {\em positive} rational linear combinations of realizable tropical varieties.  See~\cite{BabaeeHuh, Yu} for counterexamples.

\section{Connectivity}
\label{sec:connectivity}

A pure dimensional polyhedral complex is {\em connected through codimension one} if every pair of facets is connected by a path through ridges and facets of the complex.  Let $\kk$ be an algebraically closed field.  Tropical varieties of prime ideals in $\kk[x_1,\dots,x_n]$ are connected through codimension one \cite{BJSST, CartwrightPayne}. 

Let $T_1$ and $T_2$ be tropical varieties connected through codimension one.  Then the stable intersection $T_1 \stint T_2$, or even transverse intersection, need not be connected through codimension one, as we will see in Example~\ref{ex:disconnected} below. 
However, for constant coefficient tropical varieties (i.e.\ tropical varieties with respect to the trivial valuation) of prime ideals, stable intersection preserves connectivity through codimension one.  A version of the following result appeared in our earlier paper \cite{tropRes}.  This answers the last open question in \cite{CartwrightPayne} affirmatively for constant coefficient tropical varieties of irreducible varieties. 

\begin{theorem}
\label{thm:connected}
Let $I_1, I_2, \dots, I_k \subset \kk[x_1,\dots,x_n]$ be prime ideals, where $2 \leq k \leq n$, and let $\cT(I_1), \cT(I_2), \dots, \cT(I_k)$ be their constant coefficient tropical varieties respectively.  Then the stable intersection $\cT(I_1) \stint \cT(I_2) \cdots \cT(I_k)$ is itself the tropical variety of a prime ideal and thus is connected through codimension one.
\end{theorem}

\begin{proof}
First consider the case when $k = 2$.
Using the diagonal trick, we can reduce to the case when $I_1$ is a prime ideal and $I_2 = \langle x_1 - \alpha \rangle$ for a generic constant $\alpha$.  By the discussion above Lemma~\ref{lem:slicingbinomial}, we may either take $\alpha$ to be a general element in $\kk$ or to be transcendental over $\kk$; both cases give the same tropical varieties.

In general, for a prime ideal $I$ and a generic $\alpha$, specializing a variable $x_1$ to $\alpha$ may not preserve primality, i.e.\ the ideal $J_1:=I + \langle x_1 - \alpha\rangle\subseteq\overline{\kk(\alpha)}[x_1,\dots,x_n]$ need not be prime.  However, we claim that all its irreducible components have the same tropical variety.  To see this, note that $\cT(J_1)=\{0\} \times \cT(J_2)$  where $J_2$ is the ideal generated by $I$ in $\overline{\kk(x_1)}[x_2,\dots,x_n]$.  Let $J_3$ be the ideal generated by $I$ in $\kk(x_1)[x_2,\dots,x_n]$; it is prime because primality is preserved under localization.  

Whether a point is in a tropical variety can decided by computing the saturation of the ideal with $x_1\cdots x_n$.  This can be done using Gr\"obner basis computations which are not affected by taking extensions of the coefficient field, so $J_2$ and $J_3$ have the same tropical variety; hence
$$\{0\} \times \cT(J_3)=\{0\}\times \cT(J_2)=\cT(J_1).$$  Furthermore,  since $J_3$ is prime, all irreducible components of $J_2$ have the same tropicalization by \cite[Proposition~4]{CartwrightPayne}.  Since the tropical varieties of the irreducible components of $J_3$ are the same as those of the irreducible components of $J_2$, the conclusion follows.

We have shown that $\supp(\cT(I_1) \stint \cT(I_2))$ is the support of the tropical variety of a prime ideal. The cases for $k > 2$ follow by induction.
\end{proof}

In particular, it follows that the stable intersection of constant coefficient tropical hypersurfaces is connected through codimension one.  This fact has been used to for computing stable intersections of tropical hypersurfaces via fan traversal in Gfan.

The result from Theorem~\ref{thm:connected} cannot be extended to the non-constant coefficient case. To construct counterexamples, consider $k$ lattice polytopes in $k$ dimension having a regular mixed subdivision with at least two mixed cells.  Then the stable intersection of the dual tropical hypersurfaces contain at least two distinct points.  By embedding the $k$ polytopes in $\RR^n$, we can construct disconnected stable intersections of tropical hypersurfaces in $\RR^n$.

\smallskip
In general, for constant-coefficient tropical varieties of non-irreducible ideals, the stable intersection does not preserve connectivity through codimension one, as the following example shows.

\begin{example}
\label{ex:disconnected}
Let $T_1 = \cT^2(1+x_1+x_2+x_3+x_4+x_5)$ and $T_2 = \cT^2(1+x_1x_4+x_2x_5+x_3^2x_4^3+x_4^4x_5^7+x_3^6x_5)$, each of which is a stable intersection of a hypersurface with itself and hence is a $3$-dimensional fan in $\RR^5$.  They are both realizable as tropical varieties of ideals by Theorem~\ref{thm:genericvariables} and connected through codimension~$1$ by Theorem~\ref{thm:connected}.  Both $T_1$ and $T_2$ contain the two dimensional cone spanned by $-e_1$ and $-e_2$, so the union $T_1 \cup T_2$ is connected through codimension~$1$.  However for the hyperplane $H=\cT(x_1x_2x_3x_4x_5+1)$ the stable intersection $(T_1\cup T_2) \stint H$ is not connected through codimension~$1$.  To see this, we used Gfan to compute $T_1 \stint H$ and $T_2 \stint H$ using the command  {\tt gfan\_tropicalintersection --stable} and verified  that $(T_1 \stint H) \cap (T_2 \stint H) = \{0\}$ using  {\tt gfan\_fancommonrefinement}.  Thus the two dimensional fan
$(T_1 \stint H) \cup (T_2 \stint H)$ is not connected through codimension~$1$.

The fan $T_1 \cup T_2$ is realizable as the tropical variety of an ideal, but it cannot be realized as the tropical variety of a {\em prime} ideal, for any choice of multiplicities.
\qed
\end{example}

\appendix

\section{Proofs of some results in Section~\ref{sec:basics}}

Here we present detailed, careful proofs of some results needed to derive the dimension formula, balancing condition, and associativity of stable intersection.  Although elementary, these proofs, especially that of Lemma~\ref{lem:iterate}, are some of the most difficult in this paper.  Many subtle and intricate details need to be worked out.

We start by recalling the setting and some definitions from Section~\ref{sec:basics}.
Let $X$ be a tropical cycle in $N_\RR$ and $A: N \rightarrow N'$ be a linear map between lattices, inducing a linear map $A : N_\RR \rightarrow N'_\RR$. We can endow $A(X)$ with a polyhedral structure such that the image of each face of $X$ is a union of faces of $A(X)$.  For any point $\omega \in A(X)$ lying in the relative interior of a facet, let
\begin{equation}
\tag{\ref{eqn:ST}}
\mult{\omega}{A(X)} = \sum_{v} \mult{v}{X} \cdot [N'_\omega : A N_v ],
\end{equation}
where the sum runs over one $v$ for each facet of $X$ meeting the preimage of $\omega$.

\newtheorem*{lem:balanced}{Lemma~\ref{lem:balanced}}
\begin{lem:balanced}
 Let $\tau$ be a ridge in $X$ such that $A(\tau)$ also has codimension $1$ in $A (\link_X(\tau))$.  Then $A (\link_X(\tau))$ is balanced with multiplicity defined in~(\ref{eqn:ST}).
\end{lem:balanced}

\begin{proof}
Let $\tau$ be a ridge in $X$ such that $A(\tau)$ also has codimension $1$ in $A(\link_X(\tau))$.
From the balancing condition on $X$ at $\tau$, we have
$$
\sum_{\sigma \supset \tau} \mult{\sigma}{X}\cdot v_{\sigma/\tau}  \in \Span_\QQ(N_\tau).
$$
Applying the map $A$ gives
$$
\sum_{\sigma \supset \tau} \mult{\sigma}{X}\cdot A v_{\sigma/\tau}  \in \Span_\QQ(A{N_\tau}).
$$
Observe that $[ N'_{A \sigma} : N'_{A \tau} +\Span_\ZZ(A v_{\sigma/\tau})] v_{A\sigma/A\tau} \equiv A v_{\sigma/\tau}$ (mod $\Span_\QQ(AN_\tau)$) and
\begin{align*}
[N'_{A\sigma} : A N_\sigma] &=[N'_{A\sigma} :A(N_\tau+\Span_\ZZ(v_{\sigma/\tau}))] \\
&=[N'_{A\sigma}:AN_\tau+\Span_\ZZ(Av_{\sigma/\tau})]\\ &=[N'_{A\sigma}:N'_{A\tau}+\Span_\ZZ(Av_{\sigma/\tau})][N'_{A\tau}:AN_\tau].
\end{align*}
Hence 
\begin{align*}
&\sum_{\sigma \supset \tau} \mult{\sigma}{X} \cdot [N'_{A \sigma} : A N_\sigma] \cdot v_{A \sigma / A \tau}\\ =
[N'_{A\tau}:AN_\tau]&\sum_{\sigma \supset \tau} \mult{\sigma}{X} \cdot [N'_{A\sigma}:N'_{A\tau}+\Span_\ZZ(Av_{\sigma/\tau})] \cdot v_{A \sigma / A \tau}\\ \equiv
[N'_{A\tau}:AN_\tau]&\sum_{\sigma \supset \tau} \mult{\sigma}{X} Av_{\sigma/\tau}\equiv 0\hspace{3cm}\textup{(mod } \Span_\QQ(AN_\tau)).
\end{align*}
This proves that the image of a neighborhood of $\tau$ is balanced with the multiplicities given by formula~(\ref{eqn:ST}).
\end{proof}

\newtheorem*{lem:hyperplane}{Lemma \ref{lem:hyperplane}}
\begin{lem:hyperplane}
Let  $X$ be a tropical cycle and $H$ be a tropical cycle whose support is an affine hyperplane, both with positive multiplicities.  Then $X \stint H$ is also a tropical cycle, possibly zero,  with $\codim(X \stint H) = \codim(X)+1$.
\end{lem:hyperplane}

\begin{proof} 

From Lemma~\ref{lem:computingStableIntersections} we get $\dim(X\stint H)\leq \dim(X)-1$.  

If $X$ is contained in a finite union of hyperplanes parallel to $H$, then it is clear from the definition that $X \stint H$ is empty (hence a zero cycle).  

Suppose $X$ is not contained in finitely many hyperplanes parallel to $H$.  We wish to prove that every point $\omega$ in $X \stint H$ is contained in a face of dimension $\dim(X)-1$ in $X \stint H$.  By taking links if necessary, by Lemma~\ref{lem:linklineality}, we may assume that $\omega = 0$ and that $X$ is a fan and $H$ is a hyperplane through the origin. Also assume that $H$ has multiplicity $1$.

Choose a maximal cone of $X$ which is not contained in $H$ and let $S$ be its affine span. Let $U$ be a complementary subspace of $S \cap H$ in $H$, so $U \subset H$, $\dim(U) = \codim(X)$, and $S+U=N_\RR$. 

Since the multiplicities are positive and $X \minksum U$ contains a full dimensional cone, we have $X \minksum U = N_\RR$ with positive multiplicities.  Consider cones of the form $\sigma +U$ where $\sigma$ is a maximal cone of $X$ and $\sigma+U$ is full-dimensional.  These cones must cover all of $N_\RR$, and  they also cover all of $H$ in particular.  Hence there is a maximal cone $\sigma$ of $X$ such that $\dim(\sigma + U)=n$ and $\dim((\sigma + U)\cap H)=n-1$.  Since $\dim(\sigma+U) = n$, we also have $\dim(\sigma + H) = n$, so $\sigma \cap H \subset X \stint H$ by definition.  Since $\sigma + U$ is full dimensional inside $H$, and $\dim(U) = n - \dim(X) = n-1 - (\dim(X) - 1),$ we must have $\dim(\sigma \cap H) \geq \dim((\sigma+U)\cap H)-\dim(U)=n-1-\dim(U)=\dim(X)-1$ as desired.  This completes the proof that $X \stint H$ has expected codimension.

We will now check that $X \stint H$ is balanced.  By taking links at a ridge of $X \stint H$ and quotienting out by the lineality space parallel to the ridge, we can assume that $X$ is a two dimensional fan. For a generic vector $v$, $H + v$ intersects $X$ transversely, and $X \stint H$ consists of unbounded rays of $X \cap (H + v)$ counted with multiplicity.  To show the balancing of the rays of $X \stint H$, following the ideas leading to equation~(\ref{eqn:perturbedMult}), it suffices to show balancing at every vertex of $X \cap (H + v)$.

By taking links at vertices of $X \cap (H+v)$, we only have to consider the case when $X$ is a two dimensional fan with one-dimensional lineality space $\tau$ and $H$ does not contain $\tau$. In this case, for any facet $\sigma$ in $X$, $\sigma \cap H$ is a ray in $X \stint H$.  Let $v_\sigma$ be the primitive lattice vector in the ray $\sigma \cap H$.  Then by the definition of multiplicities in stable intersections, we have
$$
\mult{\sigma \cap H}{X \stint H} = \mult{\sigma}{X} \cdot [N : N_H + N_\sigma] = \frac{\mult{\sigma}{X} \cdot [N : N_H + N_\tau]}{[N_\sigma : \Span_\ZZ(v_\sigma) + N_\tau]}.
$$
On the other hand, from since $X$ is balanced, we have
$$
\sum_{\sigma \supset \tau} \mult{\sigma}{X} \cdot v_{\sigma/\tau} \in \Span_\QQ(N_\tau).
$$
Moreover we have $v_\sigma - [N_\sigma : \Span_\ZZ(v_\sigma) + N_\tau] \cdot v_{\sigma/\tau} \in \Span_\QQ(N_\tau)$.  Combining the last two statements and multiplying through with $[N : N_H + N_\tau]$ gives
$$
\sum_{\sigma \supset \tau} \mult{\sigma\cap H}{X\stint H}\cdot v_{\sigma} \in \Span_\QQ(N_\tau) \cap H = \{0\}.
$$
This shows that $X \stint H$ is balanced.
\end{proof}
 
\newtheorem*{lem:iterate}{Lemma \ref{lem:iterate}}
\begin{lem:iterate}
Let $X$ be an arbitrary tropical cycle with positive multiplicities.  Suppose $Y$ is a tropical cycle of codimension $r$ whose support is an affine linear space such that $Y = ((H_1 \stint H_2) \cdots H_r)$ where $H_1,\dots,H_r$ are tropical cycles with positive multiplicities whose supports are affine hyperplanes. Then
$$
X \stint Y = (((X \stint H_r)\stint H_{r-1})\cdots H_1).
$$
In particular, it follows that $X\stint Y$ is a tropical cycle since the right hand side is a tropical cycle by Lemma ~\ref{lem:hyperplane}.
\end{lem:iterate}

\begin{proof}
We will use induction on $r$.  When $r=1$, the statement is trivial.  Suppose $r \geq 2$, and by the inductive hypothesis, we have $$
(((X \stint H_r)\stint H_{r-1})\cdots H_1) = (X \stint H_r) \stint L
$$
where $L = ((H_1\stint H_2)\cdots H_{r-1})$.  Hence $Y = L \stint H_r$.  It remains to prove that \begin{equation}
\label{eqn:toprove}
(X\stint H_r) \stint L = X \stint (H_r \stint L).
\end{equation} 
We know that $(X\stint H_r) \stint L$ is a tropical cycle because it is equal to $(((X \stint H_r)\stint H_{r-1})\cdots H_1)$ which is a tropical cycle by the previous lemma.  At this point we have not shown yet that $ X \stint (H_r \stint L)$ is a tropical cycle. 

For simplicity of notation, let $H = H_r$.
By taking links, for equality of the supports in (\ref{eqn:toprove}) it suffices to prove that $(X \stint H) \stint L$ is non-empty if and only if $X \stint (H \stint L)$ is non-empty.  
For any linear space $V$, $X \stint V$ is non-empty if and only if the projection of $X$ onto a complement $V^\perp$ of $V$ is surjective.  Let $\pi_H$, $\pi_L$, and $\pi_{HL}$ be projections from $N_\RR$ onto the $H^\perp$, $L^\perp$, and $(H \cap L)^\perp$ respectively.  Since $H + L = N_\RR$, we have $H^\perp \cap L^\perp = \{0\}$, so $(H \cap L)^\perp$ is a direct sum of $H^\perp$ and $L^\perp$, and $\pi_H + \pi_L = \pi_{HL}$.

Suppose $(X \stint H) \stint L$ is non-empty.  Then there is a cone $\sigma \in X$ such that $$\dim(\pi_H(\sigma))=\dim(H^\perp) \text{ and }\dim(\pi_L(\sigma \cap H)) = \dim(L^\perp), \text{so}$$
 \begin{align*}\dim(\pi_{HL}(\sigma)) &= \dim(\pi_H(\sigma))+\dim(\pi_L(\sigma)) \\ &= \dim(H^\perp)+\dim(L^\perp) \\ &= \dim((H \cap L)^\perp). \end{align*}
Thus $X \stint (H \stint L)$ is non-empty.  

Now suppose $X \stint (H \stint L)$ is non-empty.  Then $\pi_{HL}(X) = H^\perp + L^\perp$, so there exists a $\sigma \in X$ such that 
\begin{align*}\dim(\pi_{HL}(\sigma)) & =\dim(H^\perp)+\dim(L^\perp), \text{ and } \\ \dim(\pi_{HL}(\sigma) \cap L^\perp) & = \dim(L^\perp). \end{align*}
Then $\dim(\pi_H(\sigma))=\dim(H^\perp)$, so $\sigma \cap H$ is a cone in $X \stint H$.
We will show that $\sigma \cap H \cap L$ is a cone in $(X \stint H)\stint L $ by showing that $\pi_L(\sigma\cap H) \supset \pi_{HL}(\sigma)\cap L^\perp$, which is full-dimensional in $L^\perp$.  

Let $v \in \pi_{HL}(\sigma)\cap L^\perp$.  Let $v' \in \sigma$ such that $\pi_{HL}(v')=v$.  Then  
$$\pi_H(v') = \pi_{HL}(v') - \pi_L(v') = v - \pi_L(v') \in L^\perp,$$ but $\pi_H(v')$ is also in $H^\perp$, so $\pi_H(v') = 0$.  Hence $v' \in H$, and 
$$v = \pi_{HL}(v') = \pi_L(v')  \in \pi_L(\sigma \cap H).$$ This proves that $\dim(\pi_L(\sigma\cap H))\geq\dim(\pi_{HL}(\sigma)\cap L^\perp)=\dim(L^\perp)$. Therefore $(\sigma\cap H)\cap L$ is a face of $(X\cdot H)\cdot L$,
 proving that $(X\stint H)\stint L$ is non-empty.

We have proven that the supports of $(X\stint H) \stint L$ and $X \stint (H \stint L)$ coincide.  Since $(X\stint H) \stint L$ is pure of expected dimension, $X \stint (H \stint L)$ is as well.



To compute multiplicities, for simplicity suppose $H$ and $L$ have multiplicity $1$ everywhere.  After taking links and taking quotients, we may assume that the support of the stable intersections on both sides of (\ref{eqn:toprove}) consist only of the origin. For generic $v_2\in N_\RR$ we have by Definition~\ref{def:stableIntersection}
$$\mult{0}{(X \stint H) \stint L} = \sum_{\substack{\tau\in X\cdot H:\\\tau \cap (L+v_2)\not=\emptyset}} \mult{\tau}{X \stint H}[N:N_{L}+N_{\tau}]$$
where, for generic $v_1\in N_\RR$,
$$\mult{\tau}{X\cdot H}=\sum_{\substack{\sigma\in X:~\sigma \cap (H+v_1)\not=\emptyset \\ \textup{and }\tau=\sigma\cap H}}\mult{\sigma}X[N:N_{H}+N_\sigma].$$
Since $\mult{H\cap L}{H\cdot L}=[N:N_{H}+N_{L}]$, we  have
$$\mult{0}{X \stint (H \stint L)}= \sum_{\substack{\sigma\in X:\\ \sigma \cap ((H\cap L)+v_3)\not=\emptyset}}\mult{\sigma}{X}[N:N_{H}+N_{L}][N:N_{H\cap L}+N_\sigma]$$
for generic $v_3\in N_\RR$. 


For $X$, $H$ and $L$ fixed, let $v$ and $v_2$ be generic vectors in $N_\RR$.  
 Since $X$ is a  finite polyhedral complex, for all sufficiently small $\varepsilon > 0$, the set of facets $\sigma \in X$ such that $\sigma\cap (H+\varepsilon v) \cap (L+v_2) \neq \emptyset$ is constant. Let $v_1 = \varepsilon v$ where $\varepsilon > 0 $ is sufficiently small and  $v_3=v_1+v_2$. 
We will show that the collections of $\sigma$'s appearing in both of the multiplicity formulas above are the same (where each $\tau$ has the form $\sigma \cap H$).  For this we need
$$\sigma \cap (H+v_1)\not=\emptyset \text{ and } \sigma \cap H \cap (L+v_2)\not=\emptyset ~~\iff~~ \sigma \cap ((H\cap L)+v_1+v_2)\not=\emptyset.$$

 Since $H+L = N_\RR$, we may assume without loss of generality that $v_1 \in L$ and $v_2 \in H$.  Then we have $(H \cap L) + v_1 + v_2 = (H+v_1) \cap (L + v_2)$. Now suppose that $\sigma \cap ((H\cap L)+v_1+v_2)\not =\emptyset$. Then $\sigma \cap (H+v_1) \cap (L + v_2) \neq \emptyset$.  This is true for all sufficient small $v_1$'s and $\sigma$ is closed, so we get  $\sigma \cap (H+v_1)\not=\emptyset$ and $\sigma \cap H \cap (L+v_2)\not=\emptyset$.

 Conversely, because $\sigma \cap H \cap (L+v_2)\not=\emptyset$ and $v_2$ is generic, $\textup{dim}(\sigma\cap H)$ is at least the codimension of $H \cap (L+v_2)$  in $H$, so
$$ \dim(\sigma \cap H) \geq \dim(H)-\dim(H \cap L).$$ We assumed that $X \stint (H \stint L)$ is $0$-dimensional, and  we have shown above that the stable intersection of a tropical cycle and a linear space has the expected dimension, so $\textup{dim}(\sigma)=\textup{codim}(H\cap L)$.
Then \begin{align*}\dim(\sigma)-\dim(\sigma \cap H) & \leq\dim(\sigma)-\dim(H)+\dim(H\cap L) \\ & =\codim(H\cap L)-\dim(H)+\dim(H\cap L) \\& = \dim(N_\RR) - \dim(H) \\&=1.\end{align*}  Moreover $\sigma$ is not contained in $H$, so $\dim(\sigma)=\dim(\sigma \cap H)+1$. Combined with $\sigma \cap (H+v_1)\not=\emptyset$ and genericity of $v_2$ this shows that offsetting $(H\cap L)+v_2$ a small amount in direction $v$ will keep the intersection $\sigma\cap ((H\cap L)+v_2+\varepsilon v)$ non-empty. This completes the proof that the $\sigma$'s appearing in the sums are the same.  

Since $N_{\tau}=N_{H}\cap N_{\sigma}$, to prove that the two multiplicities are equal, it suffices to prove for subgroups $A,B,C$ of an abelian group $N$ with well defined indices:
$$[N:A+C][N:B+(A\cap C)]=[N:A+B][N:(A\cap B)+C]$$
and apply this equation to $A=N_H,B=N_L,C=N_\sigma$. All subgroups of which we take indices contain $A\cap B$. After quotienting out by $A\cap B$, we may assume that we are in the case where $A\cap B=\{0\}$.  Then 
\begin{align*}
 [N:(A\cap B)+C] &= [N:C] \\ &= [N:A+C][A+C:C] \\&= [N:A+C][A:A\cap C] \\&= [N:A+C][A+B:(A\cap C)+B] \\&= [N:A+C][N:B+(A\cap C)]/[N:A+B].
\end{align*}
The first and fourth equalities use the facts that $A \cap B = \{0\}$. For the third equality, notice that $a+C = b+C$ if and only if $a + A\cap C = b+ A \cap C$ for all $a,b \in A$.  The rest follow from standard isomorphism theorems.
\end{proof}

\newtheorem*{lem:assocLinear}{Lemma~\ref{lem:assocLinear}}
\begin{lem:assocLinear}
Let $X, L_1,$ and $L_2$ be tropical cycles with positive multiplicities, and suppose that the supports of $L_1$ and $L_2$ are affine linear spaces.  Then
$$
X \stint (L_1 \stint L_2) = (X \stint L_1) \stint L_2.
$$ 
\end{lem:assocLinear}

\begin{proof}
 For each linear space $L$ with multiplicity $1$ and codimension $r$, we can find hyperplanes $H_1,\dots,H_r$ with multiplicities $1$ so that $L = ((H_{\sigma(1)} \stint H_{\sigma(2)})\cdots H_{\sigma(r)})$ for every permutation $\sigma$ of $\{1,\dots,r\}$. To see this choose a lattice basis $B$ of $L \cap N$ and extend this to a lattice basis $B' = B \cup \{v_1,\dots,v_r\}$ of $N$.  For $i = 1,\dots,r$, let $H_i$ be the hyperplane in $N_\RR$ spanned by $B' \backslash \{v_i\}$. Then it is straightforward to see by induction on $i$ that the tropical cycle $((H_{\sigma(1)} \stint H_{\sigma(2)})\cdots H_{\sigma(i)})$ has support $H_{\sigma(1)}\cap \cdots \cap H_{\sigma(i)}$ with multiplicity $1$ for each $i=1,\dots,r$.

Without loss of generality, we may assume that both $L_1$ and $L_2$ have multiplicity $1$ everywhere.  As shown above, we can find hyperplanes $H_1,\dots,H_r$ and $H_1',\dots, H_s'$ so that $L_1 = ((H_{\sigma(1)} \stint H_{\sigma(2)})\cdots H_{\sigma(r)})$ for every permutation $\sigma$ of $\{1,\dots,r\}$ and $L_2 = ((H_1' \stint H_2')\cdots H_s')$.  By Lemma~\ref{lem:iterate} above, we have
\begin{align*}
(X \stint L_1)\stint L_2 & = ((((X \stint H_r)\cdots H_1) \stint H_s')\cdots H_1') \\
 & = X \stint ((((H_1' \stint H_2')\cdots H_s')\stint H_1) \cdots H_r)\\
 & = X \stint (L_2 \stint ((H_r \stint H_{r-1})\cdots H_1))\\
& = X \stint (L_2 \stint L_1)\\& = X \stint (L_1 \stint L_2).
\end{align*}
\end{proof}

 \bibliographystyle{amsalpha}
\bibliography{mybib}

\newcommand{\etalchar}[1]{$^{#1}$}
\def\cprime{$'$}
\providecommand{\bysame}{\leavevmode\hbox to3em{\hrulefill}\thinspace}
\providecommand{\MR}{\relax\ifhmode\unskip\space\fi MR }
\providecommand{\MRhref}[2]{%
  \href{http://www.ams.org/mathscinet-getitem?mr=#1}{#2}
}
\providecommand{\href}[2]{#2}
\begin{thebibliography}{RGST05}

\bibitem[AR10]{AllermannRau}
Lars Allermann and Johannes Rau, \emph{First steps in tropical intersection
  theory}, Math. Z. \textbf{264} (2010), no.~3, 633--670. \MR{2591823
  (2011e:14110)}

\bibitem[BH]{BabaeeHuh}
Farhad Babaee and June Huh, \emph{A tropical approach to the strongly positive
  hodge conjecture}, arXiv:1502.00299.

\bibitem[BJS{\etalchar{+}}07]{BJSST}
T.~Bogart, A.~N. Jensen, D.~Speyer, B.~Sturmfels, and R.~R. Thomas,
  \emph{Computing tropical varieties}, J. Symbolic Comput. \textbf{42} (2007),
  no.~1-2, 54--73. \MR{2284285 (2007j:14103)}

\bibitem[CP12]{CartwrightPayne}
Dustin Cartwright and Sam Payne, \emph{Connectivity of tropicalizations}, Math.
  Res. Lett. \textbf{19} (2012), no.~5, 1089--1095. \MR{3039832}

\bibitem[CTY10]{CTY}
Mar{\'{\i}}a~Ang{\'e}lica Cueto, Enrique~A. Tobis, and Josephine Yu, \emph{An
  implicitization challenge for binary factor analysis}, J. Symbolic Comput.
  \textbf{45} (2010), no.~12, 1296--1315. \MR{2733380}

\bibitem[FS97]{FultonSturmfels}
William Fulton and Bernd Sturmfels, \emph{Intersection theory on toric
  varieties}, Topology \textbf{36} (1997), no.~2, 335--353. \MR{1415592
  (97h:14070)}

\bibitem[Ham14]{Hampe}
Simon Hampe, \emph{{\tt a-tint}: a polymake extension for algorithmic tropical
  intersection theory}, European J. Combin. \textbf{36} (2014), 579--607.
  \MR{3131916}

\bibitem[Jen]{gfan}
Anders~N. Jensen, \emph{{G}fan, a software system for {G}r{\"o}bner fans and
  tropical varieties}, Available at
  \url{http://home.imf.au.dk/jensen/software/gfan/gfan.html}.

\bibitem[JY13]{tropRes}
Anders Jensen and Josephine Yu, \emph{Computing tropical resultants}, J.
  Algebra \textbf{387} (2013), 287--319. \MR{3056698}

\bibitem[Kat12]{KatzIntersection}
Eric Katz, \emph{Tropical intersection theory from toric varieties}, Collect.
  Math. \textbf{63} (2012), no.~1, 29--44. \MR{2887109}

\bibitem[Kaz03]{Kazarnovskii}
B.~Ya. Kazarnovski{\u\i}, \emph{c-fans and {N}ewton polyhedra of algebraic
  varieties}, Izv. Ross. Akad. Nauk Ser. Mat. \textbf{67} (2003), no.~3,
  23--44. \MR{1992192 (2005a:14072)}

\bibitem[McM89]{polytopeAlgebra}
Peter McMullen, \emph{The polytope algebra}, Adv. Math. \textbf{78} (1989),
  no.~1, 76--130. \MR{1021549 (91a:52017)}

\bibitem[McM93]{McMullen_Simple}
\bysame, \emph{On simple polytopes}, Invent. Math. \textbf{113} (1993), no.~2,
  419--444. \MR{1228132 (94d:52015)}

\bibitem[Mik06]{Mik}
Grigory Mikhalkin, \emph{Tropical geometry and its applications}, Proceedings
  of the International Congress of Mathematicians, Vol. II (Madrid, 2006)
  (Z\"{u}rich, Switzerland), European Mathematical Society, 2006, pp.~827--852.

\bibitem[MS15]{MaclaganSturmfels}
Diane Maclagan and Bernd Sturmfels, \emph{Introduction to {T}ropical
  {G}eometry}, Graduate Studies in Mathematics, vol. 161, American Mathematical
  Society, Providence, RI, 2015.

\bibitem[OP13]{OssermanPayne}
Brian Osserman and Sam Payne, \emph{Lifting tropical intersections}, Doc. Math.
  \textbf{18} (2013), 121--175. \MR{3064984}

\bibitem[Rau]{Rau}
Johannes Rau, \emph{Intersections on tropical moduli spaces}, arXiv:0812.3678.

\bibitem[RGST05]{RGST}
J{\"u}rgen Richter-Gebert, Bernd Sturmfels, and Thorsten Theobald, \emph{First
  steps in tropical geometry}, Idempotent mathematics and mathematical physics,
  Contemp. Math., vol. 377, Amer. Math. Soc., Providence, RI, 2005,
  pp.~289--317. \MR{2149011 (2006d:14073)}

\bibitem[ST08]{SturmfelsTevelev}
Bernd Sturmfels and Jenia Tevelev, \emph{Elimination theory for tropical
  varieties}, Math. Res. Lett. \textbf{15} (2008), no.~3, 543--562. \MR{2407231
  (2009f:14124)}

\bibitem[Stu02]{Sturmfels02}
Bernd Sturmfels, \emph{Solving systems of polynomial equations}, CBMS Regional
  Conference Series in Mathematics, vol.~97, Published for the Conference Board
  of the Mathematical Sciences, Washington, DC, 2002. \MR{1925796
  (2003i:13037)}

\bibitem[Tve74]{Tverberg}
Helge Tverberg, \emph{How to cut a convex polytope into simplices}, Geometriae
  Dedicata \textbf{3} (1974), 239--240. \MR{0348630 (50 \#1127)}

\bibitem[Yu]{Yu}
Josephine Yu, \emph{Algebraic matroids and realizability of tropical varieties
  up to scaling}, arXiv:1506.01427.

\bibitem[Zie95]{Zie}
G{\"u}nter~M. Ziegler, \emph{Lectures on polytopes}, Graduate Texts in
  Mathematics, vol. 152, Springer-Verlag, New York, 1995. \MR{1311028
  (96a:52011)}

\end{thebibliography}

 \end{document}